\documentclass[dvipsnames]{amsart}

\usepackage[T1]{fontenc}
\usepackage[utf8]{inputenc}

\usepackage{amsmath}
\usepackage{amssymb}
\usepackage{amsthm}
\usepackage{physics}  
\usepackage{upgreek}
\usepackage{mathtools,xcolor}    
\usepackage[final]{microtype}

\usepackage{tikz}
\usepackage{tikz-cd}
\usetikzlibrary{arrows, babel, 
  decorations.text, patterns, 
  patterns.meta, decorations.pathreplacing}
\tikzset{commutative diagrams/.cd,
  arrow style=tikz, diagrams={>=stealth}}

\usepackage{url}
\usepackage[
            pdfnewwindow=false,
            colorlinks=true,
            allcolors=black,
            urlcolor=red,
            citecolor=green!50!black]{hyperref}
\usepackage[safeinputenc, 
            backend=biber, 
            url=false,
            doi=false,
            isbn=false,
            bibencoding=auto,
            style=alphabetic, 
            giveninits=true,
            backref=true]{biblatex}

\DeclareFieldFormat[article, incollection, unpublished]{title}{\mkbibemph{#1}} 
  
\AtEveryBibitem{%
   \ifentrytype{misc}{%
      \clearfield{year}%
   }{}%
}

\DefineBibliographyStrings{english}{%
  backrefpage = {$\uparrow$},
  backrefpages = {$\uparrow$},}

\AtBeginBibliography{\small}
\renewbibmacro{in:}{}

\usepackage{CJKutf8}
\begin{CJK*}{UTF8}{goth}
\gdef\yama{\mbox{\textbf{山}}}
\gdef\ten{\mbox{\textbf{天}}}
\end{CJK*}

\usepackage{longtable}
\usepackage{booktabs}
\usepackage{caption}
\usepackage{multirow}

\usepackage{enumitem}
\setenumerate[1]{label={{\bf \arabic*}.}, itemsep=3pt}
\setenumerate[2]{label=(\alph*)}
\numberwithin{equation}{section}
\numberwithin{table}{section}

\usepackage{fancyhdr}

 \usepackage[includehead, includefoot, 
  left=3cm, right=3cm, top=2cm, bottom=3cm, 
  footskip=5mm]{geometry}

\theoremstyle{plain}
\newtheorem{theorem}{Theorem}[section]
  \newtheorem{proposition}[theorem]{Proposition}
  \newtheorem{lemma}[theorem]{Lemma}
  \newtheorem{corollary}[theorem]{Corollary}
  \newtheorem{conjecture}[theorem]{Conjecture}

\newtheorem*{theorem*}{Theorem}
\newtheorem{introtheorem}{Theorem}

  \newtheorem{introconjecture}[introtheorem]{Conjecture}

\theoremstyle{remark}
\newtheorem{remark}[theorem]{Remark}
\newtheorem*{remark*}{Remark}

\newcommand{\FF}{\mathbb{F}}
\newcommand{\NN}{\mathbb{N}}

\newcommand{\PP}{\mathbb{P}}
\newcommand{\Q}{\mathbb{Q}}
\newcommand{\QQ}{\mathbb{Q}}

\newcommand{\RR}{\mathbb{R}}
\newcommand{\YY}{\mathbb{Y}}

\newcommand{\ZZ}{\mathbb{Z}}

\newcommand{\Qbar}{\overline{\mathbb{Q}}}

\renewcommand{\O}{\mathcal{O}}
\newcommand{\OO}{\mathcal{O}}
\newcommand{\OK}{\mathcal{O}_K}
\newcommand{\pp}{\mathfrak{p}}
\newcommand{\bmu}{\mathbf{\upmu}}
\newcommand{\mf}{\mathfrak}

\newcommand{\fin}{\mathsf{fin}}
\newcommand{\unr}{\textnormal{\sffamily unr}}
\newcommand{\ssimp}{\textnormal{\sffamily ss}}

\newcommand{\LCM}{\textnormal{\sffamily L}}

\DeclareMathOperator{\Aut}{\textnormal{\textsf{Aut}}}
\DeclareMathOperator{\End}{\textnormal{\textsf{End}}}

\DeclareMathOperator{\Gal}{\textnormal{\textsf{Gal}}}
\DeclareMathOperator{\GL}{\textnormal{\textsf{GL}}}

\DeclareMathOperator{\Jac}{\textnormal{\textsf{Jac}}}
\DeclareMathOperator{\lcm}{\textnormal{\textsf{lcm}}}
\DeclareMathOperator{\ord}{\textnormal{\textsf{ord}}}
\DeclareMathOperator{\Out}{\textnormal{\textsf{Out}}}

\DeclareMathOperator{\Spec}{\textnormal{\textsf{Spec}}}
\let\trace\relax
\DeclareMathOperator{\trace}{\textnormal{\textsf{tr}}}

\newcommand{\dual}[1]{{#1}^{\vee}}
\newcommand{\Sage}{{\tt SageMath\ }}

\title{Heavenly elliptic curves over quadratic fields}

\author[C.~McLeman]{Cam McLeman}
\address{{\sloppy Mathematics Department \\ University of Michigan-Flint \\ Flint MI, 48502 \\ United States}}
\email{mclemanc@umich.edu}

\author[C.~Rasmussen]{Christopher Rasmussen}
\address{Department of Mathematics and Computer Science \\ Wesleyan University \\ Middletown CT, 06459 \\ United States}
\email{crasmussen@wesleyan.edu}

\begin{document}

\begin{abstract}
  An abelian variety $A/K$ is heavenly at $\ell$ if the extension $K(A[\ell^\infty])/K(\bmu_{\ell^{\infty}}\!)$ is both pro-$\ell$ and unramified away from $\ell$.
It is known that for a fixed quadratic field $K$, the number of $K$-isomorphism classes of heavenly elliptic curves is finite, even running over all primes $\ell$. We prove a complementary result, that for a fixed prime $\ell\geq 7$, there are only finitely many such classes, even running over all quadratic fields. This naturally raises the question of whether to expect a finiteness result when both $K$ and $\ell$ are allowed to vary.  We demonstrate similarities in the behavior of heavenly elliptic curves and elliptic curves with complex multiplication, in terms of their Frobenius traces modulo $\ell$. We determine the complete list of heavenly elliptic curves defined over quadratic fields with complex multiplication and with irrational $j$-invariant (up to isomorphism). We include various extensions of our results to higher degree fields and higher-dimensional abelian varieties where possible.
\end{abstract}

\begin{CJK*}{UTF8}{goth}

\maketitle

\section{Introduction}\label{sec:intro}

Once and for all, fix an algebraic closure $\Qbar$ of $\QQ$. In what follows, all algebraic extensions of $\QQ$ will be viewed as subfields of $\Qbar$. For any $N > 1$, we let $\bmu_N$ denote the set of $N$th roots of unity in $\Qbar$ and set $\bmu_{N^\infty} = \bigcup_i \bmu_{N^i}$. If $K$ is a number field and $\ell$ is a rational prime, we define\footnote{The kanji $\ten$ is pronounced ``ten'' and has the meaning \emph{heaven}. See \S \ref{sec:IharaQuestion} for an explanation of this notation.} $\ten := \ten(K, \ell)$ to be the maximal pro-$\ell$ extension of $K(\bmu_{\ell^\infty}\!)$ which is unramified away from $\ell$. We call an abelian variety $A/K$ \emph{heavenly} at $\ell$ if $K(A[\ell^\infty]) \subseteq \ten$. Heavenly abelian varieties arise in the study of the pro-$\ell$ outer Galois representation attached to $\PP^1_K-\{0,1,\infty\}$, and their arithmetic relates to an open question of Ihara on the kernel of this representation (see \S\ref{sec:IharaQuestion}). It has been conjectured \cite[Conj.~1]{Rasmussen-Tamagawa:2008} that for any number field $K$ and any $g > 0$, there are only finitely many pairs $(\ell, [A]_{K})$ where $A/K$ is a $g$-dimensional abelian variety which is heavenly at $\ell$.

\subsection{New results}

The present paper establishes three main results. The first is a finiteness result on heavenly elliptic curves when the prime $\ell$ is fixed but the field of definition $K$ is allowed to vary over all quadratic fields.

\begin{theorem}[Corollary \ref{cor:HQ21_finite}]\label{IntroThm1}
  Let $\ell\geq 7$ be prime. The collection of elliptic curves which are defined over a quadratic field and which are heavenly at $\ell$ represent only finitely many $\Qbar$-isomorphism classes.
\end{theorem}
In fact, Theorem \ref{IntroThm1} is a specific case of a stronger result.
\begin{theorem}[Theorem \ref{thm:Hkd1l_finite}]
  Fix a number field $K_{0}$ and a number $d > 1$. Let $\mathcal{K}$ be the set of fields $K$ of degree at most $d$ over $K_{0}$. Let $\ell$ be any prime satisfying $\ell > 2d + 1$. The collection of elliptic curves which are defined over a field in $\mathcal{K}$ and which are heavenly at $\ell$ represent only finitely many $\Qbar$-isomorphism classes.
\end{theorem}

The second new result describes the behavior of the trace of Frobenius for heavenly elliptic curves defined over quadratic fields.

\begin{theorem}[Corollary \ref{cor:balanced_ap_mod_ell}]\label{IntroThm2}
  Let $K$ be a quadratic field and $\ell > 19$ a prime number. Suppose $E/K$ is an elliptic curve which is heavenly at $\ell$. Suppose $\mathfrak{p}$ is a prime of $K$ not dividing $\ell$, with norm $p^{f}$. Then
  \begin{enumerate}[left=0pt, nosep, label={(\alph*)}]
  \item If $\mathfrak{p}$ is inert in $K(\sqrt{-\ell})$, then $a_{\mathfrak{p}}(E) \equiv 0 \bmod{\ell}$.
    \item If $\mathfrak{p}$ splits in $K(\sqrt{-\ell})$, then $a_{\mathfrak{p}}(E)^{2} \equiv 4p^{f} \bmod{\ell}$. \qedhere
  \end{enumerate}
\end{theorem}
This is strikingly similar to well-known results for elliptic curves with complex multiplication (see Proposition \ref{prop:ap_mod_ell_for_CM}). Building off of this similarity, we characterize, among elliptic curves with complex multiplication defined over quadratic fields, those which are heavenly  (Theorem \ref{thm:heavenly_iff_nonsurj_trace}). Using this we search for all elliptic curves which are defined over a quadratic field, have irrational $j$-invariant, possess complex multiplication, and are heavenly at some prime $\ell$. This gives our third result.
\begin{theorem}[Theorem \ref{thm:heavenly_cm_calculation}]\label{IntroThm3}
  Let $K$ be a quadratic field and $\ell$ a rational prime. Suppose $E/K$ has complex multiplication, is heavenly at $\ell$, and has $j(E) \not\in \QQ$. Then $E$ is isomorphic, over $K$, to one of the curves listed in Table \ref{table:all_non_rational_j_curves}. Furthermore, every curve listed in Table \ref{table:all_non_rational_j_curves} is indeed heavenly.
\end{theorem}

\subsection{Conjectures on heavenly elliptic curves}

We now make three related conjectures on heavenly elliptic curves over quadratic fields and explain how our new results, together with existing results in the literature, support the conjectures. The first conjecture is natural in light of Theorem \ref{IntroThm2}. 

\begin{introconjecture}\label{conj:balanced_implies_cm}
Fix a prime $\ell>19$ and a quadratic number field $K$. If $E/K$ is heavenly at $\ell$, then $E$ has complex multiplication.
\end{introconjecture}

To prove Conjecture \ref{conj:balanced_implies_cm}, one need ``only'' improve the inert case of Theorem \ref{IntroThm2} to the equality $a_{\mathfrak{p}} = 0$ (see Remark \ref{rmk:ConverseForCM}), but at present this seems quite out of reach.

The other conjectures are finiteness statements regarding heavenly elliptic curves defined over quadratic number fields. Let $\mathcal{F}$ denote the set of quadratic fields contained in $\Qbar$, ordered by the absolute value of their discriminant (and with ties broken arbitrarily). Let $\mathcal{R} \subseteq \NN \times \mathcal{F}$ be the set of pairs $(\ell, K)$ for which
\begin{enumerate}[label={(\alph*)}, nosep]
  \item $\ell\geq 7$ is prime,
  \item there exists $E/K$ heavenly at $\ell$, with
  \item $E \times_{K} \Qbar \not\cong E_0 \times_{\QQ} \Qbar$ for any $E_0/\QQ$ which is heavenly at $\ell$.
\end{enumerate}
For any fixed quadratic field $K$, there are only finitely many $K$-isomorphism classes of heavenly elliptic curves defined over $K$ \cite[Prop.~7.4]{Rasmussen-Tamagawa:2017}. Consequently, for each $K \in \mathcal{F}$ there is some constant $\ell_{K}$, dependent on $K$, with the property that if $E/K$ is heavenly at some prime $\ell$, then $\ell < \ell_{K}$. Thus, no points of $\mathcal{R}$ lie below the boundary labeled {\sffamily R.-Tamagawa} in Figure \ref{fig:H21_map}.
Analogously, Theorem \ref{IntroThm1} implies that for each fixed $\ell \geq 7$, there are only finitely many $K \in \mathcal{F}$ with $(\ell, K) \in \mathcal{R}$, and so no points of $\mathcal{R}$ lie above the analogous boundary marked {\sffamily Theorem \ref{IntroThm1}.}

\begin{introconjecture}\label{conj:R_is_finite}
  The set $\mathcal{R}$ is finite.
\end{introconjecture}

In \cite{Bourdon:2015}, Bourdon established a finiteness result for heavenly elliptic curves with complex multiplication over arbitrary number fields, and demonstrated a uniform bound for $\ell$ in terms of the degree of the field of definition only. Her result implies any elliptic curve defined over a quadratic field which is heavenly at a prime $\ell > 163$ cannot possess complex multiplication. Taken together, Bourdon's result and Conjecture \ref{conj:balanced_implies_cm} would improve \cite[Prop.~7.4]{Rasmussen-Tamagawa:2017} to a uniform result, implying that every pair $(\ell, K) \in \mathcal{R}$ satisfies $\ell \leq 163$. In light of Theorem \ref{IntroThm1}, it follows that Conjecture \ref{conj:balanced_implies_cm} implies Conjecture \ref{conj:R_is_finite}.

The final conjecture is the analogous finiteness statement about isomorphism classes of heavenly elliptic curves. Let $\overline{\mathcal{H}}$ be the set of pairs $([E \times_{K} \Qbar]_{\Qbar}, \ell)$ where
\begin{enumerate}[label={(\alph*)}, nosep]
\item $\ell \geq 7$ is prime and $K \in \mathcal{F}$,
\item $E$ is an elliptic curve defined over $K$ which is heavenly at $\ell$, with
\item $E \times_{K} \Qbar \not\cong E_0 \times_{\QQ} \Qbar$ for any $E_0/\QQ$ which is heavenly at $\ell$. 
\end{enumerate}

\begin{introconjecture}\label{conj:H_is_finite}
  The set $\overline{\mathcal{H}}$ is finite.
\end{introconjecture}

If $A/K$ is heavenly at $\ell$, then necessarily $A$ has good reduction away from $\ell$ \cite[Thm.~1]{Serre-Tate:1968}. By the resolution of the Shafarevich Conjecture \cite{Faltings:1983, Zarhin:1985}, there can be only finitely many $K$-isomorphism classes of elliptic curves over a fixed $K$ with good reduction away from $\ell$. Thus, Conjecture \ref{conj:R_is_finite} implies Conjecture \ref{conj:H_is_finite}. A delicate argument is required to demonstrate the converse, as one must demonstrate any infinite family $\{E_{i}/K_{i} \}$ of heavenly elliptic curves lying in one $\Qbar$-isomorphism class must necessarily also be represented by a heavenly elliptic curve over $\QQ$. The equivalence of Conjectures \ref{conj:R_is_finite} and \ref{conj:H_is_finite} is proved in \cite{McLeman-Rasmussen:2025}.

\begin{figure}[ht!]
\centering
\begin{tikzpicture}[>=stealth, scale=0.15]

\draw[fill=red!15, red!15] (0,0) rectangle (5, 50);
\node[fill=red!15] at (2.5, 20) {\rotatebox{90}{\small \textcolor{red}{\sffamily infinitely many classes}}};

\draw[->, thick, black] (5,0) -- (5,51);
  
\node[blue, anchor=south east] at (10.4, 0.3) {\rotatebox{90}{\small \sffamily Conjecture A}};

\draw[fill=blue!15, blue!25] (27,0) rectangle (50,50);
\node[blue, anchor=south east] at (27.4, 0.3) {\rotatebox{90}{\small \sffamily Bourdon (2015)}};

  \draw[->, thick, black, fill=gray!40,
  , postaction={decorate, decoration={raise=1.0ex, text along path, text align={left indent={0.65\dimexpr\pgfdecoratedpathlength\relax}}, text={|\small\sffamily|R.-Tamagawa}}}] 
  (50,0) -- (34, 0) .. controls (34, 10) and (36, 30) .. (50.4, 40);

  \draw[->, thick, black, fill=gray!40, postaction={decorate, decoration={raise=-2.0ex, text along path, text align={left indent={0.73\dimexpr\pgfdecoratedpathlength\relax}}, text={|\small\sffamily|Theorem 1.1}}} ] (5, 50) -- (5, 28) .. controls (24, 28) and (36, 42) .. (40, 50.4);

  

  \draw[->, thick, dashed, blue] (10,0) -- (10,51);

  \draw[->, thick, blue] (27, 0) -- (27, 51);

  
  \draw[thick, black] (0,0) -- (18,0);
  \draw[thick, black] (20,0) -- (30,0);
  \draw[->, thick, black] (32,0) -- (51,0) node[right] {$\ell$};

  \foreach \xcoord/\xlabel in {1/2, 2/3, 4/5, 6/7, 10/19, 27/163} {
    \draw[black, fill] (\xcoord, 0) circle (0.2) node[below] {\tiny $\xlabel$}; 
  }


  \draw[thin, black] (18, 1) --
      ++(0.25, -0.25) -- ++(-0.25, -0.25) --
      ++(0.25, -0.25) -- ++(-0.25, -0.25) --
      ++(0.25, -0.25) -- ++(-0.25, -0.25) --
      ++(0.25, -0.25) -- ++(-0.25, -0.25);

  \draw[thin, black] (20, 1) --
      ++(0.25, -0.25) -- ++(-0.25, -0.25) --
      ++(0.25, -0.25) -- ++(-0.25, -0.25) --
      ++(0.25, -0.25) -- ++(-0.25, -0.25) --
      ++(0.25, -0.25) -- ++(-0.25, -0.25);

  \draw[thin, black] (30, 1) --
      ++(0.25, -0.25) -- ++(-0.25, -0.25) --
      ++(0.25, -0.25) -- ++(-0.25, -0.25) --
      ++(0.25, -0.25) -- ++(-0.25, -0.25) --
      ++(0.25, -0.25) -- ++(-0.25, -0.25);

  \draw[thin, black] (32, 1) --
      ++(0.25, -0.25) -- ++(-0.25, -0.25) --
      ++(0.25, -0.25) -- ++(-0.25, -0.25) --
      ++(0.25, -0.25) -- ++(-0.25, -0.25) --
      ++(0.25, -0.25) -- ++(-0.25, -0.25);
  
  \draw[->, thick, black] (0, 0) -- (0, 51) node[above] {$\mathcal{F}$};
\end{tikzpicture}
\caption{\small The set $\mathcal{R}$ inside $\NN \times \mathcal{F}$}\label{fig:H21_map}
\end{figure}

\subsection{Outline}

In \S\ref{sec:arithAV} we review previous results used in the rest of the paper. In \S\ref{sec:RTconj} we prove Theorem \ref{IntroThm2} using Levin's generalization of Siegel's theorem on $S$-integral points on curves. In \S\ref{sec:balanced}, \ref{sec:bal_bounds}, we demonstrate that (uniformly and dependent only on the dimension $g$ and the degree of $K$), abelian varieties over $K$ which are heavenly at a prime $\ell \gg 0$ possess an additional property: the structure of their $\ell$-torsion, viewed as a group scheme, is tightly controlled (see the \emph{balanced} condition in \S\ref{sec:balanced}). We prove that this balanced property holds for heavenly elliptic curves over quadratic fields whenever $\ell > 19$ (in addition to many cases when $\ell\leq 19$; see Proposition \ref{prop:B21_is_19}).  In \S\ref{sec:frob_trace_heavenly}, we demonstrate the similarity in trace behavior between heavenly elliptic curves and elliptic curves with complex multiplication. This is used to prove Theorem \ref{IntroThm3} in \S\ref{sec:heavenly_with_cm}. A few additional results on heavenly elliptic curves, independent of the presence of complex multiplication, are presented in \S\ref{sec:non_cm}. 

\subsection*{Acknowledgments}
The authors are grateful to Akio Tamagawa for many helpful discussions that improved the breadth of this work. The authors also appreciate helpful correspondence with John Voight and John Cremona regarding the LMFDB database \cite{LMFDB:2024}. The authors thank the anonymous referee for a thorough and constructive report. In addition to several helpful suggestions improving the exposition of the paper, the referee suggested a computational approach that led to a strengthening of Theorem \ref{thm:heavenly_cm_calculation}; the supporting MAGMA script \cite{McLeman-Rasmussen:2026} is a direct implementation of this idea. 

\section{Arithmetic of abelian varieties}\label{sec:arithAV}

\subsection{General notions}\label{subsec:general_notions}

For any number field $K$, let $G_K := \Gal(\Qbar/K)$ and let $\OK$ denote the ring of integers of $K$. For a fixed prime $\ell$, let $\chi \colon G_{\QQ} \to \FF_\ell^\times$ denote the $\ell$-adic cyclotomic character modulo $\ell$, defined by the condition  $\zeta^\sigma = \zeta^{\chi(\sigma)}$ for all $\sigma \in G_{\QQ}$ and all $\zeta \in \bmu_\ell$. Suppose $L/K$ is a finite extension and $\mf{P}$ is a prime of $L$ above $\mf{p}$. We let $e_{\mf{P}\mid\mf{p}}$ denote the ramification index for $\mf{P}$ over $\mf{p}$ and let $f_{\mf{P}\mid\mf{p}} = [\O_L/\mf{P} : \O_K/\mf{p}]$. In case $K = \QQ$, we write $e_{\mf{P}}$ and $f_{\mf{P}}$ for these quantities.

If $A/K$ is an abelian variety and $m \geq 1$, let $\rho_{A,m} \colon G_K \to \GL_{2g}(\ZZ/m\ZZ)$ denote the $G_K$-representation attached to the group $A[m]$ of $m$-torsion points of $A$, with respect to some choice of basis. The image of $\rho_{A,m}$, up to conjugacy, is independent of this choice of basis. 

If $\pp$ is a prime of $K$, let $\mathbf{N}\mf{p}$ denote the absolute norm of $\mf{p}$, let $K_{\pp}$ denote the completion of $K$ at $\mf{p}$, and $\overline{K}_{\pp}$ denote a (fixed) algebraic closure of $K_{\pp}$. We let $K_{\pp}^{\unr}$ denote the maximal unramified extension of $K_{\mf{p}}$ inside $\overline{K}_{\mf{p}}$, and let $I_{\mf{p}} := \Gal(\overline{K}_{\mf{p}}/K_{\mf{p}}^{\unr})$. We let $\theta_{\mf{p}} \in G_K$ denote any choice of Frobenius for $\mf{p}$ (unique up to conjugacy and the action of inertia). If $E/K$ is an elliptic curve with good reduction at $\pp$, set $a_{\pp} = a_{\pp}(E) := \mathbf{N}\pp + 1 - \sharp E(\FF_{\pp})$. As is well known, $a_{\pp} \equiv \trace \rho_{E,\ell}(\theta_{\pp}) \bmod{\ell}$. Since $\rho_{E,\ell}$ factors through a finite quotient of $G_K$, the Chebotarev Density Theorem guarantees every element of the quotient is represented by a Frobenius for a prime of good reduction; consequently if $E$ has conductor $\mf{N}$, $\trace \rho_{E,\ell}(G_K) = \{ a_{\pp}(E) \bmod \ell : \pp \nmid \ell \mf{N} \}$.

When we say an elliptic curve $E/K$ ``has complex multiplication,'' we mean potential complex multiplication; we do not assume all endomorphisms are defined over $K$.

\subsection{The invariant \texorpdfstring{$e_A(\mf{l})$}{eA(l)}}\label{subsec:constructions-over-local-fields}

Suppose $F$ is a local field whose residue field has characteristic $\ell$. Fix an algebraic closure $\overline{F}$. Let $F^\unr$ be the maximal unramified extension of $F$ inside $\overline{F}$, and set $I := \Gal(\overline{F}/F^\unr)$. Suppose $J \leq I$ is a finite index subgroup such that $\overline{F}\vphantom{F}^J/F^\unr$ is a tamely ramified extension. Let $\O$ be the ring of integers of $\overline{F}\vphantom{J}^J$ and let $\mf{m}$ be the maximal ideal of $\O$. Fix a uniformizer $\pi \in \mf{m}$ and let $\xi \in \overline{F}$ be a root of the polynomial $X^{\ell-1} - \pi$. The ratio of any two roots of this polynomial gives an element of $\bmu_{\ell-1} \subseteq \O$. For each $\sigma \in J$, the element $\psi(\sigma) \in \FF_\ell^\times$ is the unique choice satisfying $\psi(\sigma) \equiv \xi^{\sigma}/\xi \bmod{\mf{m}}$.
This defines a surjective homomorphism $\psi \colon J \to \FF_\ell^\times$, called the \emph{fundamental character} on $J$. 

Now, suppose $A/F$ is an abelian variety. Among subfields of $\overline{F}$ which contain $F^{\unr}$, let $F^{\ssimp}$ be the minimal field over which $A \times_F F^{\ssimp}$ has semistable reduction. The \emph{index of semistable reduction} for $A$ over $F$ is $e(A/F) := [F^{\ssimp} : F^{\unr}]$.

Given a number field $K$, a prime $\mf{l}$ of $K$ above $\ell$, and an abelian variety $A/K$, set $A_{K_{\mf{l}}} := A \times_K K_{\mf{l}}$. We define related invariants $e(A/K; \mf{l}) := e(A_{K_\mf{l}}/K_{\mf{l}})$ and $e_A(\mf{l}) := e(A/K; \mf{l}) \cdot e_{\mf{l}}$. When no confusion will occur, we write $e$ for $e_A(\mf{l})$.

\subsection{The totient identity}\label{sec:totid}

Suppose $A/K$ is an abelian variety and $p$ is a rational prime. Set $T_pA := \varprojlim\limits A[p^n]$, the $p$-adic Tate module of $A$, and set $V_pA := T_pA \otimes_{\ZZ_p} \QQ_p$. 
Let $0 = V_0 \subseteq V_1 \subseteq \cdots \subseteq V_m = V_p A$ be a $G_K$-stable filtration of $V_pA$ of maximum possible length. We set $\tilde{V}_{p}A := \bigoplus_{i=1}^{m} (V_{i}/V_{i-1})$, the semisimplification of $V_pA$. The Galois representation associated to $\tilde{V}_{p}A$, denoted $\tilde{\rho}_{A,p^{\infty}}$, is block diagonal with respect to an obvious choice of basis. 

If $p \neq \ell$, let $\widetilde{\rho} := \widetilde{\rho}_{A_{K_\mf{l}},p^\infty}$ be the representation associated to $\widetilde{V}_{p}A_{K_{\mf{l}}}$. We set $J_{\mf{l}} := I_{\mf{l}} \cap \ker \widetilde{\rho}$ and $K_{\mf{l}}^{\ssimp} :=  \overline{K}_{\mf{l}}^{J_{\mf{l}}}$. We recall some classical results from \cite[Expos\'{e} IX]{SGA:1972}: $K_{\mf{l}}^{\ssimp}$ is the minimal extension of $K_{\mf{l}}^{\unr}$ over which $A_{K_{\mf{l}}}$ obtains semistable reduction. Consequently, $J_{\mf{l}}$ does not depend on the choice of prime $p \neq \ell$. Moreover, provided $\ell > 2g + 1$, the extension $K_{\mf{l}}^{\ssimp}/K_{\mf{l}}^{\unr}$ is tamely ramified and the quotient $M := I_{\mf{l}}/J_{\mf{l}}$ is cyclic and $\sharp M = [I_{\mf{l}} : J_{\mf{l}}] = e(A/K; \mf{l})$. 

We denote the fundamental character on $J_{\mf{l}}$ by $\psi_{\mf{l}}$. Serre established the following relationship between $\chi$ and $\psi_{\mf{l}}$.
\begin{proposition}[{\cite[\S1, Prop.~8]{Serre:1972}}]\label{prop:chi_is_psi_e}
The restriction of $\chi$ to $J_{\mf{l}}$ satisfies $\eval{\chi}_{J_{\mf{l}}} = \psi_{\mf{l}}^e$.
\end{proposition}

In \cite{Rasmussen-Tamagawa:2017}, the action of $M$ on $\widetilde{V}_{p} A_{K_{\mf{l}}}$, together with the decomposition of the group algebra $\QQ[M]$ as a product of cyclotomic fields, is used to obtain the following constraint on $e(A/K; \mathfrak{l})$. 

\begin{proposition}[Prop.~6.2, Cor.~6.4, \cite{Rasmussen-Tamagawa:2017}]\label{prop:tot_id}
Suppose $A/K$ is an abelian variety of dimension $g > 0$, $\ell$ is a rational prime satisfying $\ell > 2g + 1$, and $\mf{l}$ is a prime of $K$ above $\ell$. Let $\varphi$ denote Euler's totient function. There exist nonnegative integers $\{ n_d : d \mid e(A/K;\mf{l}) \}$ such that
\begin{equation}\label{eqn:tot_id}
2g = \sum_d n_d \varphi(d), \qquad e(A/K;\mf{l}) = \lcm \{ d : n_d > 0 \}, \quad 2 \mid n_1, \quad 2 \mid n_2.
\end{equation}
\end{proposition}

We use this to bound $e(A/K; \mf{l})$ in terms of $g$ only. For any finite list $\mathbf{d} = (d_1, \dots, d_m)$ of positive integers, let $\LCM(\mathbf{d})$ denote $\lcm(d_1, \dots, d_m)$ and define $\Phi(\mathbf{d}) := \sum_{i=1}^m \varphi(d_i)$. For each $g \geq 1$, we define
\begin{equation}\label{eqn:defn_tilde_e}
    \tilde{e}(g) := \max \{ \LCM(\mathbf{d}) : \mathbf{d}\ \text{satisfying}\ \Phi(\mathbf{d}) = 2g \}.
\end{equation}
 Since the totient function takes on any given value only finitely many times, $\tilde{e}(g)$ is well-defined. One list which satisfies $\Phi(\mathbf{d}) = 2g$ is $\mathbf{d} = (2g,1, \dots, 1)$ of length $2g + 1 - \varphi(2g)$, and so $\tilde{e}(g) \geq 2g$. In fact, it is known that $\tilde{e}(g)$ coincides with the maximum order for a torsion element of $\GL_{2g}(\ZZ)$ \cite{Suprunenko:1963}. By Proposition \ref{prop:tot_id}, we obtain the bound $e(A/K; \mf{l}) \leq \tilde{e}(g)$, provided $\ell > 2g + 1$. In case $g = 1$, it is straightforward to enumerate all possible forms of the totient identity and conclude the following.

\begin{corollary}\label{cor:g_1_e_bound}
Let $E/K$ be an elliptic curve, $\ell>3$ a prime, and $\mf{l}$ a prime of $K$ above $\ell$.  Then $e(E/K;\mf{l}) \in \{1,2,3,4,6\}$.
\end{corollary}

\subsection{The fields \texorpdfstring{$\yama$}{Yama} and \texorpdfstring{$\ten$}{Ten} and Ihara's question}\label{sec:IharaQuestion} 
We recall the field $\yama$ (``yama'') related to the arithmetic of Jacobian varieties and a long-standing question of Ihara. Set $\YY := \PP^1_K \smallsetminus \{0,1, \infty \}$ and make a compatible selection of base points $\overline{y} \in \overline{\YY}(\overline{\QQ})$, $y \in \YY(\overline{\QQ})$, and an embedding $b \colon K \to \overline{\QQ}$. Set $\overline{\pi} := \pi_1(\overline{\YY}, \overline{y})$, $\pi := \pi_1(\YY, y)$, and note $G_K \cong \pi_{1}(\Spec K, b)$. One of the guiding principles of anabelian geometry is that the arithmetic of $\overline{\QQ}$ is reflected in the geometry of \'{e}tale coverings of $\overline{\YY}$. Attached to the short exact sequence
\[ 1 \longrightarrow \overline{\pi} \longrightarrow \pi \longrightarrow G_K \longrightarrow 1, \]
there is an induced representation $\Phi_{\YY,K} \colon G_K \to \Out \overline{\pi}$. Here, we consider the pro-$\ell$ variant: as the pro-$\ell$ completion $\overline{\pi}^{(\ell)} := \pi_1^{\text{\textsf{pro-}}\ell}(\overline{\YY}, \overline{y})$ is a characteristic quotient of $\overline{\pi}$, we obtain a compatible representation $\Phi_{\YY,K,\ell} \colon G_K \to \Out \overline{\pi}^{(\ell)}$. The field $\yama := \yama(K,\ell)$ is defined to be the fixed field of the subgroup $\ker \Phi_{\YY,K,\ell} \leq G_K$. Ihara demonstrated $\yama \subseteq \ten$ in the case $K = \QQ$ and asked whether these fields coincide \cite{Ihara:1986}. Sharifi demonstrated \cite{Sharifi:2002} that when $\ell$ is an odd regular prime, the equality  $\yama(\QQ,\ell) = \ten(\QQ,\ell)$ is equivalent to the Deligne-Ihara conjecture (now a theorem of Francis Brown \cite{Brown:2012}). Ihara's full question, however, remains open. 

\begin{remark*}
The kanji $\yama$ means \emph{mountain} and the kanji $\ten$ means \emph{heaven}; the notation lets us rephrase Ihara's question as \emph{Does the mountain reach the heavens}?
\end{remark*}

\subsection{Heavenly abelian varieties}\label{sec:heavenly_AV}

Suppose $C/K$ is a complete nonsingular irreducible curve and $f \colon C \to \PP^1$ is a Galois covering of degree $\ell^n$ defined over $K$ whose branch locus is contained in $\{0,1,\infty\}$. Let $\Jac C$ denote the Jacobian of $C$. The existence of $f$ guarantees that $\Jac C$ has good reduction away from $\ell$. It is a theorem of Anderson and Ihara \cite{Anderson-Ihara:1988} that $K(\Jac(C)[\ell^\infty]) \subseteq \yama(K, \ell)$ in this case; however, $C$ need not admit such a covering in order for the containment to hold -- see \cite{Rasmussen-Tamagawa:2008, Rasmussen-Tamagawa:2019} for examples in dimensions $g = 1, 3$ respectively. Sufficient conditions for a $K$-rational cyclic covering $C \to \PP^1$ to satisfy $K(\Jac(C)[\ell^\infty]) \subseteq \ten$ also exist \cite{Rasmussen-Tamagawa:2019}.

In the context of Ihara's question, it is natural to ask how uncommon it is for an abelian variety to be heavenly at $\ell$, i.e., to satisfy $K(A[\ell^\infty]) \subseteq \ten$. The following characterization for an abelian variety to be heavenly at $\ell$ can be found in \cite{Rasmussen-Tamagawa:2008, Rasmussen-Tamagawa:2017}; we give an explicit proof here.
\begin{lemma}\label{lemma:heavenly_criteria}
Let $K$ be a number field and $A/K$ an abelian variety. Let $\ell$ be a rational prime and let $S$ be the set of primes above $\ell$ in $K$. Then $A$ is heavenly at $\ell$ if and only if
\begin{enumerate}[label={(\alph*)}, nosep, left=1em]
\item $A$ has good reduction outside $S$, and
\item $[K(A[\ell]) : K(\bmu_\ell)]$ is a power of $\ell$. \qedhere
\end{enumerate}
\end{lemma}
\begin{proof}
Suppose $A$ is heavenly at $\ell$. The extension $\ten/K$ is unramified at any finite place $\pp \not \in S$, and so the same is true for the subextension $K(A[\ell^n])/K$  for any $n \geq 1$. By \cite{Serre-Tate:1968} we conclude $A$ has good reduction at any $\mf{p} \not\in S$ and so (a) holds. Further, (b) holds since $K(A[\ell])/K(\bmu_\ell)$ is a finite subextension of $\ten/K(\bmu_\ell)$.

Conversely, assume (a) and (b). By (a), \cite{Serre-Tate:1968} implies $K(A[\ell^\infty])/K$ is unramified at every finite $\mf{p} \not\in S$. In particular, $K(A[\ell^\infty])/K(\bmu_{\ell^\infty})$ is unramified at all places not in $S$. The extension $K(A[\ell^\infty])/K(A[\ell])$ is pro-$\ell$ for any abelian variety; this follows from the fact that the kernel of the natural reduction map $\GL_N(\ZZ_\ell) \to \GL_N(\FF_\ell)$ is pro-$\ell$. Since $K(A[\ell^\infty])/K(\bmu_\ell)$ is Galois, (b) implies $K(A[\ell^\infty])/K(\bmu_\ell)$ is pro-$\ell$. Thus, $K(A[\ell^\infty]) \subseteq \ten$ and $A$ is heavenly at $\ell$.
\end{proof}

For any number field $K$, any $g > 0$, and any prime $\ell$, let $\mathcal{H}(K, g, \ell)$ denote the set of $K$-isomorphism classes of abelian varieties which are heavenly at $\ell$. Recall the conjectures of Shafarevich, proved by Faltings \cite{Faltings:1983, Faltings:1984}: The set of $K$-isomorphism classes of abelian varieties of dimension $g$ which admit a polarization of degree $d$ and which have good reduction outside a finite set $S$ of primes of $K$, is finite, for any choice of $K$, $g$, $d$, and $S$. This was generalized by Zarhin \cite{Zarhin:1985} to establish the finiteness of the set of $K$-isomorphism classes of abelian varieties of dimension $g$ with good reduction outside $S$, with no polarization condition. Consequently, Lemma \ref{lemma:heavenly_criteria} implies that $\mathcal{H}(K, g,  \ell)$ is always a finite set. Tamagawa and the second author have conjectured a finiteness result even when the prime $\ell$ varies. 

\begin{conjecture}\label{conj:RT}
    Let $K$ be a number field and $g > 0$. The set $\mathcal{H}(K,g)$ is finite. Equivalently, $\mathcal{H}(K, g, \ell) = \varnothing$ for $\ell \gg 0$. 
\end{conjecture}

\begin{remark}
The set $\mathcal{H}(K, g)$ is always nonempty, because the set $\mathcal{H}(\QQ,1)$ is nonempty, and for each $([E]_{\QQ}, \ell) \in \mathcal{H}(\QQ,1)$, we have $([(E \times_\QQ K)^g]_{K}, \ell) \in \mathcal{H}(K,g)$. 
\end{remark}

If $A/K$ is heavenly at $\ell$, then there exists a basis for $A[\ell]$ such that the associated Galois representation $\rho_{A,\ell}$ has an especially nice form. 
\begin{proposition}\label{prop:ell-tors-repn}
  Suppose $([A]_K, \ell) \in \mathcal{H}(K, g)$. For each $r$ with $1 \leq r \leq 2g$, there exists $i_r \in \ZZ$ satisfying $0 \leq i_r < \sharp \chi(G_K)$ such that 
  \begin{equation}\label{eqn:rhoAell}
  \rho_{A,\ell} \sim \begin{pmatrix}
      \chi^{i_1} & \star  & \cdots & \star \\
      & \chi^{i_2} & \cdots & \star \\
      && \ddots & \vdots \\
      &&& \chi^{i_{2g}} \end{pmatrix}.
      \end{equation}
\end{proposition}
\begin{proof}
    This is \cite[Lemma 3.3]{Rasmussen-Tamagawa:2017}.
\end{proof}

\begin{remark}\label{rmk:ir_sum}
Since $\rho_{A,\ell}$ satisfies $\det \rho_{A,\ell} = \chi^g$, it holds that $\sum_{r=1}^{2g} i_r \equiv g \bmod{(\ell-1)}$.
\end{remark}

\begin{remark}\label{rmk:heavenly_isog}
If two abelian varieties $A/K$ and $A'/K$ are $K$-isogenous, then the torsion fields $K(A[\ell^\infty])$ and $K(A'[\ell^\infty])$ coincide. Thus, ``heavenly at $\ell$'' is an isogeny-invariant property. Moreover, for any abelian variety $A/K$ with complex multiplication by some order of a field $L$, necessarily there exists $A'/K$, $K$-isogenous to $A$, with complex multiplication by the maximal order of $L$ \cite[Lemma 2.4]{Lombardo:2018}. Consequently, when considering whether an abelian variety $A/K$ with complex multiplication by an order $\O$ of $L$ is heavenly, we may without loss of generality replace $A$ with an isogenous abelian variety possessing complex multiplication by the maximal order. 
\end{remark}

\section{New finiteness results}\label{sec:RTconj}\label{sec:vertical_finiteness}
Conjecture \ref{conj:RT} has been proven under the assumption of the generalized Riemann Hypothesis \cite{Rasmussen-Tamagawa:2017}. Unconditionally, the conjecture remains open, although several partial results have been established, which we summarize here.  For any $K \subseteq \Qbar$ and any $g \geq 0$ we define
\begin{align*}
    \mathcal{H}^{\textsf{CM}}(K, g) & := \bigl\{ ([A]_K, \ell) \in \mathcal{H}(K, g) : A\ \text{has complex multiplication by an order in $K$} \bigr\}, \\
    \mathcal{H}^{\textsf{pot-CM}}(K, g) & := \bigl\{ ([A]_K, \ell) \in \mathcal{H}(K, g) : A\times_K \overline{K}\ \text{has complex multiplication} \bigr\}, \\
    \mathcal{H}^{\textsf{ab}}(K, g) & := \bigl\{ ([A]_K, \ell) \in \mathcal{H}(K, g) : \rho_{A,\ell} \text{ has abelian Galois image} \bigr\}, \\
    \mathcal{H}^{\textsf{QM}}(K, 2) & := \bigl\{ ([A]_K, \ell) \in \mathcal{H}(K, 2) : \End A\ \text{contains an indefinite division algebra over $\QQ$}, \bigr\} \\
     \mathcal{H}^{\textsf{semist}}(K, g) & := \bigl\{ ([A]_K, \ell) \in \mathcal{H}(K, g) : A\ \text{has semistable reduction over}\ k \bigr\}. 
\end{align*}
Then the following sets are known (unconditionally) to be finite:
\begin{itemize}[itemsep=0.25em, left=1ex, label={$\cdot$}]
\item $\mathcal{H}(K, 1)$ for $[K:\QQ] \leq 3$ \cite{Rasmussen-Tamagawa:2008, Rasmussen-Tamagawa:2017},
\item $\mathcal{H}(K, 1)$ for extensions $K/\QQ$ whose Galois group has exponent $3$ \cite{Rasmussen-Tamagawa:2017},
\item $\mathcal{H}(\QQ,g)$ for $g \leq 3$ \cite{Rasmussen-Tamagawa:2017},   
\item $\mathcal{H}^{\mathsf{CM}}(K, g)$, for all $(K, g)$ \cite{Ozeki:2013},
\item $\mathcal{H}^{\mathsf{ab}}(K, g)$, for all $(K, g)$ \cite{Ozeki:2013},
\item $\mathcal{H}^{\mathsf{QM}}(K, 2)$, for $K$ an imaginary quadratic field with $h_K > 1$ \cite{Arai-Momose:2014},
\item $\mathcal{H}^{\textsf{pot-CM}}(K, g)$, for all $(K, g)$ \cite{Bourdon:2015} ($g=1$), \cite{Lombardo:2018} (arbitrary $g$),
\item $\mathcal{H}^{\mathsf{semist}}(K, g)$, for all $K$ and $g$ \cite{Rasmussen-Tamagawa:2017}.
\end{itemize}

If Conjecture \ref{conj:RT} holds, then $H(K, g) := \sup \{\ell : \mathcal{H}(K,g,\ell) \neq \varnothing \}$ is well-defined. In this context, it is also reasonable to consider questions of uniformity, e.g., does there exist a constant $H(d, g)$ such that $H(K, g) \leq H(d, g)$ provided $[K:\QQ] = d$? See \cite{Bourdon:2015, Rasmussen-Tamagawa:2017} for some such uniform results.

Alternatively, one may fix $g$ and $\ell$ and consider finiteness questions where the field of definition varies. More precisely, define $\overline{\mathcal{H}}(K, g, \ell) := \{ [A]_{\Qbar} : [A]_K \in \mathcal{H}(K, g, \ell) \}$, and set
\[ \overline{\mathcal{H}}_{K}(d, g, \ell) := \bigcup_{L \colon [L:K] \leq d} \overline{\mathcal{H}}(L,g,\ell).
\]
The set $\overline{\mathcal{H}}_K(d, g, \ell)$ is not necessarily finite; for example $\overline{\mathcal{H}}_\QQ(2, 1, 2)$ is infinite \cite[Prop.~5.4]{Rasmussen-Tamagawa:2017}. However, finiteness does hold for $\ell$ sufficiently large.

\begin{theorem}\label{thm:Hkd1l_finite}
Suppose $d > 1$ and $\ell$ is a prime satisfying $\ell - 1 > 2d$. For any number field $K$, $\sharp \overline{\mathcal{H}}_K(d, 1, \ell) < \infty$.
\end{theorem}

We let $M_K$ denote the set of places of $K$, and let  $M_K^\infty$ and $M_K^\fin$ denote the subsets of archimedean and nonarchimedean places, respectively. Throughout, $S$ will be a finite subset of $M_K$ satisfying $M_K^\infty \subseteq S$. Let $\OO_{K, S}$ denote the ring of $S$-integers in $K$ and let $\overline{\OO}_{K, S}$ denote the integral closure of $\OO_{K, S}$ in $\Qbar$. For any finite extension $K'/K$, we set $S \times_K K' := \{ w \in M_{K'} : \eval{w}_K \in S \}$. For simplicity, we set $\OO_{K', S} := \OO_{K, S \times_K K'}$.

Suppose $Y/K$ is a nonsingular affine curve and $S$ is a finite set of places of $K$ that includes all archimedean places. Let $K'/K$ be a fixed finite extension. We set
\[ \mathcal{I}_{S, d}(Y; K') := \{ y \in Y(\overline{\OO}_{K', S}) : [K'(y) : K'] \leq d \} = \bigcup_{L \colon [L : K'] \leq d}
Y(\OO_{L,S}).  \]
That is, $\mathcal{I}_{S,d}(Y; K')$ is the set of all points of $Y$ of degree at most $d$ over $K'$, which are integral outside of the primes above $S$. 

It is well-known that when sufficiently many points ``at infinity'' are removed from a complete projective curve, the open subset that remains admits only finitely many $S$-integral points. More precisely:
\begin{theorem}\label{thm:Levin}
Suppose $Y/K$ is a nonsingular affine curve and $X/K$ is a nonsingular projective completion of $Y$. Let $S$ be a finite set of places of $K$ which includes all the infinite places, and let $d \geq 1$ be fixed. If $\sharp(X(\overline{K})-Y(\overline{K})) > 2d$, then for any finite extension $K'/K$, the set $\mathcal{I}_{S,d}(Y; K')$ is finite.
\end{theorem}
In case $d = 1$, this theorem is due to Siegel \cite{Siegel:1929, Fuchs:2014}. For general $d$ the result is due to Levin \cite{Levin:2016, Levin:2009}, generalizing earlier work of Corvaja and Zannier in the case $d = 2$ \cite[Cor.~1]{Corvaja-Zannier:2004}. In fact, Levin demonstrates a necessary and sufficient condition for the set $\mathcal{I}_{S,d}(Y; K')$ to be infinite: the existence of an open subset $Y' \subseteq Y$ and a finite morphism $Y' \to \mathbb{G}_m$ of degree at most $d$. Since no such morphism can exist if $\sharp (X(\overline{K}) - Y(\overline{K})) > 2d$, Theorem \ref{thm:Levin} follows.

For any $\ell$, let $Y_1(\ell)/\QQ$ be the affine modular curve which parametrizes isomorphism classes of pairs $(E, P)$, where $E$ is an elliptic curve and $P \in E[\ell]$. Let $X_1(\ell)$ be the standard compactification of $Y_1(\ell)$. We have the following consequence of Theorem \ref{thm:Levin}.

\begin{corollary}\label{cor:X1_ell_finiteness}
Suppose $K$ is a number field and $d \geq 2$. Set $S_0 := \{\ell, \infty \} \subseteq M_\QQ$ and $S := S_0 \times_\QQ K$. If $\ell - 1 > 2d$, then $\sharp \mathcal{I}_{S,d}(Y_1(\ell); K) < \infty$.
\end{corollary}

\begin{proof}
 The curve $X_1(\ell)$ admits $\ell-1$ cusps corresponding to pairs $(C, \zeta)$ where $C$ is a N\'{e}ron polygon and $\zeta$ is an $\ell$-th root of unity lying on some component of $C$ (see \cite[\S9.3]{Diamond-Im:1995} for details). Thus, $\sharp (X_1(\ell)(\overline{K}) - Y_1(\ell)(\overline{K})) = \ell - 1 > 2d$ and by Theorem \ref{thm:Levin}, the set $\mathcal{I}_{S,d}(Y_1(\ell); K)$ is finite.
\end{proof}

\begin{proof}[Proof of Theorem \ref{thm:Hkd1l_finite}]
Set $S_0 := \{\ell, \infty \}$ and $S := S_0 \times_\QQ K$. It will suffice to demonstrate an injective map of sets $\Psi \colon \overline{\mathcal{H}}_{K}(d, 1, \ell) \to \mathcal{I}_{S,d}(Y_1(\ell), K(\bmu_\ell))$, since the codomain is finite by Corollary \ref{cor:X1_ell_finiteness}. Suppose $\gamma \in \overline{\mathcal{H}}_K(d, 1, \ell)$. There exists $L$ with $[L:K] \leq d$ and $E/L$ an elliptic curve with $[E]_L \in \mathcal{H}(L, 1, \ell)$ such that $\gamma = [E]_{\Qbar}$. Since $E$ is heavenly at $\ell$, $E/L$ has good reduction outside of $S \times_K L$ by Lemma \ref{lemma:heavenly_criteria}. By Proposition \ref{prop:ell-tors-repn}, the Galois representation attached to $E[\ell]$ has the form
\[ \rho_{E, \ell} \sim \begin{pmatrix*} \chi^i & \star \\ 0 & \chi^{1-i} \end{pmatrix*}. \]
Set $K' := K(\bmu_\ell)$ and $L' := K'L = L(\bmu_\ell)$. Note $[L' : K'] \leq [L : K] \leq d$. The restriction of $\rho_{E, \ell}$ to $G_{L'}$ is necessarily unit upper triangular (up to conjugacy). Consequently, there exists nontrivial $\ell$-rational points in $E(L')$.  Choose such a point $P$.  Under the moduli interpretation, $[(E, P)]$ is an $L'$-rational point of $Y_1(\ell)$. Since $E$ has good reduction outside $S$,
\[ y := [(E,P)] \in Y_1(\ell)(\OO_{L',S}) \subseteq \mathcal{I}_{S', d}(Y_1(\ell), K'), \qquad S' := S \times_K K'. \]
The map $\Psi$ is defined by $\gamma \mapsto y$. The map $\Psi$ must be injective, for if $[(E, P)] = [(E', P')]$ as points of $Y_1(\ell)$, then there exists an isomorphism $E \rightarrow E'$ defined over $\Qbar$, and consequently $[E]_{\Qbar} = [E']_{\Qbar}$. The finiteness of $\overline{\mathcal{H}}_K(d, 1, \ell)$ now follows.
\end{proof}

\begin{corollary}\label{cor:HQ21_finite}
The set $\overline{\mathcal{H}}_{\QQ}(2,1,\ell)$ is finite for any prime $\ell \geq 7$.
\end{corollary}

\begin{remark}
As mentioned above, $\overline{\mathcal{H}}_{\QQ}(2,1,\ell_0)$ is infinite when $\ell_0 = 2$, and it is natural to consider the finiteness in case $\ell_0 \in \{3, 5 \}$. In case $\ell_0 = 5$, there exists a morphism $Y_1(5) \to \mathbb{G}_m$ of degree $2$ which we expect can be used to parametrize an infinite family of quadratic $S$-integral points on $Y_1(5)$, and consequently, an infinite subset of $\overline{\mathcal{H}}_{\QQ}(2,1,5)$. The situation around $\ell_0 = 3$ is more subtle (roughly due to the fact that $Y_1(3)$ cannot be viewed as a fine moduli space).
\end{remark}

\section{Balanced abelian varieties}\label{sec:Tate-Oort-numbers}\label{sec:balanced}

Suppose $A/K$ is an abelian variety of dimension $g$, heavenly at $\ell$, and let $\mf{l}$ be a prime of $K$ above $\ell$. Proposition \ref{prop:chi_is_psi_e} gives one relationship between the characters $\chi$ and $\psi_{\mf{l}}$, but a second critical relation exists as well. As $\rho_{A,\ell}$ has the form \eqref{eqn:rhoAell}, the classification of Tate and Oort \cite{Tate-Oort:1970} now implies the following. For each $r$ with $1 \leq r \leq 2g$, there exists an integer $j_{\mf{l},r}$ with $0 \leq j_{\mf{l}, r} \leq e$ such that $\chi^{i_r} = \psi_{\mf{l}}^{j_{\mf{l},r}}$ on $J_{\mf{l}}$ -- further detail is given in  \cite[\S3.3]{Rasmussen-Tamagawa:2017}. We call the set of integers $\{ j_{\mf{l}, r} \}$ the \emph{Tate-Oort numbers} of $A$ at $\mf{l}$. They are uniquely determined if $e < \ell - 1$; otherwise they are determined modulo $(\ell-1)$.

We say $A$ is \emph{weakly balanced} at $\mf{l}$ if there exists a partition of $\{1, 2, \dots, 2g\}$ into $g$ unordered pairs, say $R_1 \sqcup R_2 \sqcup \cdots \sqcup R_g$, such that for each $R_k = \{s, t \}$, we have $j_{\mf{l},s} + j_{\mf{l},t} = e$. We say $A$ is \emph{balanced} at $\mf{l}$ if $j_{\mf{l},r} = \frac{e}{2}$ for every $r$, $1 \leq r \leq 2g$. If $A$ is balanced at every $\mf{l}$ above $\ell$, we say $A$ is \emph{balanced} at $\ell$.

Now on $J_{\mf{l}}$ we have both $\chi^{i_r} = \psi_{\mf{l}}^{j_{\mf{l}, r}}$ and $\chi = \psi_{\mf{l}}^{e}$. Consequently, for each $r$ we have $\chi^{i_r e} = \psi_{\mf{l}}^{j_{\mf{l}, r} e} = \chi^{j_{\mf{l}, r}}$, which implies
\begin{equation}\label{eqn:i_j_reln}
e i_r \equiv j_{\mf{l}, r} \pmod{\ell-1}, \qquad 1 \leq r \leq 2g. 
\end{equation}

The congruence \eqref{eqn:i_j_reln} forces a condition on the trace of Frobenius elements under $\rho_{A,\ell}$. Suppose $A/K$ is heavenly at $\ell$ and let $p \neq \ell$ be a rational prime. Suppose $\pp$ is a prime of $K$ above $p$ and set $q := \mathbf{N}\mf{p}$. Let $\theta_\mf{p} \in G_K$ be a Frobenius element for $\mf{p}$. For any $m \geq 1$, let $P_m(T) \in \ZZ[T]$ denote the characteristic polynomial for the action of $\theta_{\mf{p}}^m$ on $V_\ell A$. Let $\{\alpha_r \}_{r=1}^{2g}$ be the complex roots of $P_1(T)$. For any $m \geq 1$,
\begin{equation}\label{eqn:frob_char_poly}
 P_m(T) = \prod_{r=1}^{2g} (T - \alpha_r^m) = T^{2g} - \tau_m T^{2g-1} + \cdots + q^{gm}, \quad \tau_m := \alpha_1^m + \cdots + \alpha_{2g}^m. 
\end{equation}
Consider how $\theta_{\mf{p}}^m$ acts on $A[\ell]$. The characteristic polynomial of this action, $\overline{P}_m(T) \in \FF_\ell[T]$, coincides with $P_m(T)$ modulo $\ell$. By Proposition \ref{prop:ell-tors-repn}, $\rho_{A,\ell}(\theta_\mf{p}^m)$ is upper triangular with respect to some basis. Consequently, $\overline{P}_m(T)$ must split completely over $\FF_\ell$, and the multiset $\{ \chi^{i_r}(\theta_{\mf{p}}^m) : 1 \leq r \leq 2g \}$ gives the roots of $\overline{P}_m$ counted with multiplicity. Thus,
\[ \overline{P}_m(T) = \prod_{r=1}^{2g} (T - \chi^{i_r}(\theta_{\mf{p}}^m)) = \prod_{r=1}^{2g} (T - \chi(\theta_{\mf{p}})^{mi_r}) = \prod_{r=1}^{2g} (T - q^{mi_r}). \]
In the special case $m = e$, \eqref{eqn:i_j_reln} implies
\begin{equation} 
\overline{P}_{e}(T) = \prod_{r=1}^{2g} (T - q^{ei_r}) = \prod_{r=1}^{2g} (T - q^{j_{\mf{l},r}}).
\end{equation}
By comparing the coefficients of $T^{2g-1}$ in $\overline{P}_e$ and $P_e$, we obtain:
\begin{equation}\label{eqn:tau_e_congruence}
\tau_{e} - \sum_{r=1}^{2g} q^{j_{\mf{l},r}} \equiv 0 \pmod{\ell}. 
\end{equation}
Note \eqref{eqn:tau_e_congruence} represents an infinite family of congruences satisfied by the Tate-Oort numbers $j_{\mf{l},r}$ (one for each $\pp \nmid \ell$). For sufficiently large $\ell$, we will show these congruences  \emph{almost} force a violation of the Weil conjectures -- the only exception being the case when $j_{\mf{l},r} = \frac{e}{2}$ for every $r$, i.e., when $A$ is balanced at $\mf{l}$.

Let $\mathcal{B}(K,g)$ denote the subset of $\mathcal{H}(K,g)$ of pairs $([A]_{K}, \ell)$ where $A$ is balanced at $\ell$. We define
\begin{align*}
B(K, g) & := \sup\, \{\ell : \text{there exists}\ ([A],\ell) \in \mathcal{H}(K,g) \smallsetminus \mathcal{B}(K, g) \}, \\
B(n, g) & := \sup\, \{ B(K, g) : [K:\QQ] = n \}.
\end{align*}
Thus, if a $g$-dimensional abelian variety $A/K$ is heavenly at $\ell > B(K,g)$, then $A$ must be balanced at $\ell$. That $B(n, g)$ is finite was already implicitly established in \cite[\S3]{Rasmussen-Tamagawa:2017}, but the implied bound is very large; we give stronger explicit bounds below. 

One useful consequence of the balanced property is that it forces a divisibility condition on $e$.

\begin{lemma}\label{lemma:4_divides_e}
Suppose $\ell > 2$ and $\mf{l}$ is a prime of $K$ above $\ell$. Suppose $A/K$ is heavenly at $\ell$ and for some index $r$, $j_{\mf{l},r} = \frac{e}{2}$. Then $\ord_2 e > \ord_2 (\ell-1)$. In particular, $4 \mid e$.
\end{lemma}
\begin{proof}
Set $j = j_{\mf{l},r}$. We have $e = 2j$ and so \eqref{eqn:i_j_reln} implies $j(2i_r-1) \equiv 0 \pmod{\ell-1}$. But $\ell - 1$ is even and $2i_r - 1$ is odd. Thus $ \ord_2 e > \ord_2 j \geq \ord_2(\ell-1) > 0$.
\end{proof}

\section{Bounds on balanced heavenly abelian varieties}\label{sec:bal_bounds}

In this section, we determine explicit bounds for $B(K,g)$. We begin by showing that heavenly abelian varieties are always weakly balanced.

\begin{theorem}\label{thm:heavenly_implies_wb}
Suppose $A/K$ is heavenly at $\ell$, and suppose $\mf{l}$ is a prime of $K$ above $\ell$. Then $A/K$ is weakly balanced at $\mf{l}$.
\end{theorem}

\begin{proof}
The content of the theorem is trivial in case $\ell = 2$, so we assume $\ell > 2$. Suppose $A/K$ has dimension $g$. We have $e = e(A/K; \mf{l}) \cdot e_{\mf{l}}$. Recall that
\[ \QQ_\ell^\unr \subseteq K_{\mf{l}}^\unr \subseteq K_{\mf{l}}^\ssimp, \quad [K_{\mf{l}}^\unr : \QQ_\ell^\unr] = e_{\mf{l}}, \quad [K_{\mf{l}}^\ssimp : K_{\mf{l}}^\unr] = e(A/K; \mf{l}), \]
and that the Tate-Oort numbers $j_r$ may be chosen with $0 \leq j_r \leq e$. Set $K_0 := K \cap \QQ(\bmu_\ell)$ and consider the tower $\QQ \subseteq K_0 \subseteq \QQ(\bmu_\ell)$. Set
\begin{equation}\label{eqn:def_ep_w}
\varepsilon := [K_0 : \QQ], \quad w := [\QQ(\bmu_\ell) : K_0] = \sharp \chi(G_K), \quad b := \frac{e}{\varepsilon}.
\end{equation}
Note $\varepsilon \cdot w = \ell - 1$. Since $\QQ(\bmu_\ell)/\QQ$ is totally ramified over $\ell$, the same is true for $K_0/\QQ$. As $K_0 \subseteq K$, it follows that $\varepsilon \mid e_{\mf{l}}$ and so $b \in \ZZ$. 

As $\rho_{A,\ell}$ is upper triangular with respect to some basis, there exists a filtration on $A[\ell]$,
\[ 0 = U_0 \leq U_1 \leq \cdots \leq U_{2g} = A[\ell], \qquad \dim_{\FF_{\ell}} U_r = r. \]
Let $\dual{A}$ be the dual abelian variety of $A$. Since the Weil pairing $\varepsilon_\ell \colon A[\ell] \times \dual{A}[\ell] \to \FF_\ell(\chi) \cong \bmu_\ell$ is perfect, there is an analogous filtration induced on $\dual{A}[\ell]$:
\[ 0 = \dual{U}_{2g} \leq \dual{U}_{2g-1} \leq \cdots \leq \dual{U}_0 = \dual{A}[\ell], \qquad \dim_{\FF_\ell} \dual{U}_r = 2g - r. \]
Subspaces of dimension $r$ inside $A[\ell]$ correspond to subspaces of codimension $r$ inside $\dual{A}[\ell]$. Consequently,
\[ \rho_{\dual{A},\ell} \sim 
    \begin{pmatrix} 
      \chi^{1 - i_{2g}} & \star & \cdots & \star \\
       & \chi^{1-i_{2g-1}} & \cdots & \star \\
       && \ddots & \vdots \\
       &&& \chi^{1-i_1} \\ \end{pmatrix}. \]

Define the multisets $X := \{ \chi^{i_r} : 1 \leq r \leq 2g \}$ and $\dual{X} := \{ \chi^{1-i_r} : 1 \leq r \leq 2g \}$ (where here $\chi$ is implicitly restricted to $G_{K}$).  We claim that $X=\dual{X}$ (as multi-sets).

For any $\sigma\in G_K$, let $\gamma := \rho_{A,\ell}(\sigma)$ and $\gamma^\vee := \rho_{\dual{A}, \ell}(\sigma)$. Let $Q$, $Q^\vee \in \FF_{\ell}[T]$, be the characteristic polynomials for $\gamma$, $\gamma^{\vee}$, respectively. Since both representations are upper-triangular,
\[ Q(T) = \prod_{r=1}^{2g} (T - \chi^{i_r}(\sigma)), \qquad \dual{Q}(T) = \prod_{r=1}^{2g} (T - \chi^{1-i_r}(\sigma)). \]
Let $Q_\ell$ and $Q^\vee_\ell$ be the characteristic polynomials for the respective actions of $\sigma$ on $T_{\ell}(A)$ and $T_{\ell}(A^{\vee})$; consequently also on $V_\ell A$ and $V_\ell\dual{A}$. Necessarily, $Q \equiv Q_{\ell}$ and $Q^\vee \equiv Q^\vee_{\ell}$ modulo $\ell$.

Since $A$ and $\dual{A}$ are isogenous, there is an induced isomorphism of $\QQ_\ell[G_K]$-modules $V_\ell A \cong V_\ell \dual{A}$. Hence $Q_\ell=Q^\vee_\ell$ and so $Q = Q^\vee$. We conclude that for all $\sigma\in G_K$ we have the equality of multisets $\{\chi^{i_r}(\sigma)\}=\{\chi^{1-i_r}(\sigma)\}$. Choose $\delta\in G_K$ whose image in $\Delta := \Gal(K(\bmu_\ell)/K)$ generates $\Delta$. Since $\chi$ factors through $\Delta$, $\{\chi^{i_r}(\delta)\}=\{\chi^{1-i_r}(\delta)\}$ implies $X = \dual{X}$. 

Since $0 \leq i_r < w$ by Proposition \ref{prop:ell-tors-repn}, this  gives an equality of multisets of ``exponents modulo $w$,''
\[ \{i_r \bmod{w} : 1 \leq r \leq 2g \} = \{1 - i_r \bmod{w} : 1 \leq r \leq 2g \}. \]
Scaling each class by $e$, we obtain $\{e i_r \bmod{ew} \}_{r=1}^{2g}  = \{ e(1-i_r) \bmod{ew} \}_{r=1}^{2g}$. But $(\ell - 1) \mid ew$, and so
\begin{align*}
    \{j_r \bmod{(\ell-1)} : 1 \leq r \leq 2g \}
       & = \{e i_r \bmod{(\ell-1)} : 1 \leq r \leq 2g \} \\
       & = \{ e(1-i_r) \bmod{(\ell-1)} : 1 \leq r \leq 2g \} \\
       & = \{ e - j_r \bmod{(\ell-1)} : 1 \leq r \leq 2g \}.
\end{align*} 
Thus, there exists a permutation $\theta$ in the symmetric group $\mf{S}_{2g}$, satisfying $j_r \equiv e - j_{\theta(r)} \pmod{(\ell-1)}$ for any $r \in \{1, \dots, 2g \}$. Since $j_s \equiv e - j_r$ if and only if $j_r \equiv e - j_s$, we may assume $\theta$ is a disjoint product of transpositions, say $\theta = (s_1\ t_1)(s_2\ t_2) \cdots (s_h\ t_h)$. We claim $\{1, \dots, 2g \}$ can be partitioned into $g$ unordered pairs $R_k$ that each satisfy
\begin{equation}\label{eqn:wb_mod_ell_minus_1}
R_k = \{s, t \}, \quad j_s + j_t \equiv e \bmod{(\ell-1)}.
\end{equation}
If $h = g$, this is immediate by setting $R_k := \{s_k, t_k \}$ for each $k$, $1 \leq k \leq g$. If $h < g$, then there exist  indices $r$ for which $r = \theta(r)$. We have $2j_r - e \equiv 0 \bmod{(\ell-1)}$. It follows that $2 \mid e$, as both $2j_r$ and $\ell-1$ are even. Thus $\tfrac{e}{2} \in \ZZ$ and $j_r$ satisfies
\[ j_r \equiv \frac{e}{2} \bmod{(\ell-1)} \qquad \text{or} \qquad j_r \equiv \frac{e}{2} + \frac{\ell-1}{2} \bmod{(\ell-1)}. \]
Define
\[ X := \left\{ r : j_r \equiv \frac{e\vphantom{\ell}}{2} \bmod{(\ell-1)} \right\}, \qquad Y := \left\{ r : j_r \equiv \frac{e}{2} + \frac{\ell-1}{2} \bmod{(\ell-1)} \right\}, \]
and note $\{1, \dots, 2g \} = \{s_1, t_1, \dots, s_h, t_h \} \sqcup X \sqcup Y$. We show that $\sharp X$ and $\sharp Y$ must both be even. Towards a contradiction, suppose not. Then
\[ \sum_{r=1}^{2g} j_r \equiv ge + \frac{\ell-1}{2} \bmod{(\ell-1)}. \]
We also have $\sum_r i_r \equiv g \bmod{(\ell-1)}$ (Remark \ref{rmk:ir_sum}). This implies
\[ ge + \frac{\ell-1}{2} \equiv \sum_{r=1}^{2g} j_r \equiv \sum_{r=1}^{2g} ei_r \equiv e \sum_{r=1}^{2g} i_r \equiv eg \bmod{(\ell-1)}, \]
and so $\tfrac{\ell-1}{2} \equiv 0 \bmod{(\ell-1)}$, an obvious contradiction. So $\sharp X$ and $\sharp Y$ are even, and
\[ X = \{s_{h+1}, t_{h+1}, \dots, s_{h+h'}, t_{h+h'} \}, \qquad Y = \{s_{h+h'+1}, t_{h+h'+1}, \dots, s_{g}, t_{g} \}. \]
Set $R_k := \{s_k, t_k \}$ for each $k$, $1 \leq k \leq g$, and now \eqref{eqn:wb_mod_ell_minus_1} holds for every $R_k$. 

If $\ell - 1 > e$, then $-(\ell-1) < -e \leq j_s + j_t - e \leq e < (\ell - 1)$, and \eqref{eqn:wb_mod_ell_minus_1} implies $j_s + j_t = e$. 

If $\ell - 1 \leq e$, then choose $j_s'$ with $0 \leq j_s' \leq \ell - 2$ such that $j_s' \equiv j_s \bmod{(\ell-1)}$, and set $j_t' := e - j_s'$. Now
\[ 0 \leq j_s', j_t' \leq e, \quad j_t' = e - j_s' \equiv e - j_s \equiv j_t \bmod{(\ell-1)}, \quad j_s' + j_t' = e. \]
Since the Tate-Oort numbers are only defined modulo $(\ell-1)$, we may replace $\{j_s, j_t \}$ with $\{j_s', j_t'\}$ and this implies $A$ is weakly balanced.
\end{proof}

\begin{proposition}\label{prop:bal_bound}
Suppose $A/K$ is an abelian variety of dimension $g > 0$ which is heavenly at $\ell$, and suppose $\mf{l}$ is a prime of $K$ with $\mf{l} \mid \ell$. Let $\pp$ be a prime of $K$ with $\pp \nmid \ell$ and take $q = \mathbf{N}\mf{p}$. If $\ell > g(q^{e/2} + 1)^2$, then $A$ is balanced at $\mf{l}$.
\end{proposition}

\begin{proof}
Suppose $([A],\ell) \in \mathcal{H}(K,g)$ and $\ell > g(q^{e/2} + 1)^2$. Since $A$ is weakly balanced at $\mf{l}$ there exists a partition of $\{1, \dots, 2g \}$ into $g$ unordered pairs $\{R_k\}$ satisfying $\sum_{r \in R_k} j_r = e$ for each $k$, $1 \leq k \leq g$. Thus
  \[ 2q^{e/2} \leq \sum_{r \in R_k} q^{j_r} \leq q^e + q^0, \]
  which implies
\begin{equation}\label{eqn:sum_qj_ineq}
-g(q^e + 1) \leq - \sum_{r=1}^{2g} q^{j_r} = - \sum_{k=1}^g \sum_{r \in R_k} q^{j_r} \leq - 2gq^{e/2}.
\end{equation}
On the other hand, the Weil conjectures provide $|\tau_e| = \left| \sum_{r} \alpha_r^e \right| \leq 2gq^{e/2}$. With \eqref{eqn:sum_qj_ineq}, this implies a bound on $\tau_e - \sum_r q^{j_r}$:
\begin{equation}\label{eqn:tau_e_sum_qj_bound}
-\ell < -g(q^{e/2} + 1)^2 =  -g(q^e + 1) - 2gq^{e/2} \leq \tau_e - \sum_{r=1}^{2g} q^{j_r} \leq 0.
\end{equation}
Taken together, \eqref{eqn:tau_e_congruence} and \eqref{eqn:tau_e_sum_qj_bound} imply $\tau_e = \sum_r q^{j_r}$. This is possible only if $j_r = \frac{e}{2}$ for every $r$, and so $A$ is balanced at $\mf{l}$.
\end{proof}

\begin{corollary}\label{cor:bal_unif_bound}
  For any $n \geq 1$ and any $g \geq 1$, $B(n,g) \leq g(2^{n^2\tilde{e}(g)/2} + 1)^2$.
\end{corollary}

\begin{proof}
  Suppose $[K:\QQ] = n$, $\ell > g(2^{n^2\tilde{e}(g)} + 1)^2$, and $([A],\ell) \in \mathcal{H}(K,g)$. Let $\mf{l}$ be any prime of $K$ above $\ell$, and let $\mf{p}$ be the prime of $K$ above $2$ of minimal norm $q$. Since $q \leq 2^n$ and $e \leq n\tilde{e}(g)$, we have
\[ q^{e/2} \leq (2^n)^{n\tilde{e}(g)/2} = 2^{n^2 \tilde{e}(g)/2}. \]
Thus $\ell > g(2^{n^2\tilde{e}(g)/2} + 1)^2 \geq g(q^{e/2} + 1)^2$, and so $A$ is balanced at $\mf{l}$ by Proposition \ref{prop:bal_bound}. 
\end{proof}

The set $\mathcal{H}(\QQ,1)$ was determined explicitly in \cite{Rasmussen-Tamagawa:2008}, and it is straightforward to verify $B(1,1) = 11$. (This is sharp, as there do exist elliptic curves over $\QQ$ which are heavenly, but not balanced, at $11$.) For any fixed choice of $(n,g)$, it is possible to improve the bound of Corollary \ref{cor:bal_unif_bound} through a more careful analysis; we give one such result.

\begin{proposition}\label{prop:B21_is_19}
 $B(2,1) \leq 19$. Moreover, if $[K:\QQ] = 2$ and $E/K$ is an elliptic curve which is heavenly at a prime $\ell > 11$, but $E$ is not balanced at some $\mf{l} \mid \ell$, then $\ell \in \{13, 19 \}$ and $\{ j_{\mf{l},1}, j_{\mf{l},2} \} = \{0, e \}$ for some $e \in \{3, 6, 12 \}$.
\end{proposition}

\begin{proof}
 We must answer the following question: for a fixed pair $(j_1, j_2)$ of integers with $j_1 < j_2$ and $j_1+j_2=e$, for which $\ell$ could there exist $E/K$ which is heavenly at $\ell$ and has a prime $\mf{l} \mid \ell$ whose Tate-Oort numbers are $(j_1, j_2)$?  

The triple $(\ell, j_{1}, j_{2})$ must satisfy several constraints. First, by \eqref{eqn:i_j_reln}, $\gcd(e, \ell - 1) \mid j_1$. Second, there must exist integers $0\leq i_1, i_2\leq \ell - 2$ satisfying the conditions of Section \ref{sec:balanced} and Theorem \ref{thm:heavenly_implies_wb}. Third, each rational prime $p \neq \ell$ exerts a constraint: for each prime $\pp$ of $K$ above $p$, the Frobenius at $\pp$ satisfies the conditions established in \S\ref{sec:Tate-Oort-numbers}. Consequently, for every such $p$ there exist $f \in \{1, 2 \}$ and $\tau_{1}, \tau_e \in \ZZ$ such that:
\begin{align}\label{eqn:tau_constraints}
q &= p^f,\qquad j_1<j_2,\qquad j_1+j_2=e,\qquad |\tau_1| \le 2\sqrt{q}, \qquad
\tau_1 \equiv q^{i_1} + q^{i_2} \pmod{\ell}  \\
&\gcd(e, \ell - 1) \mid j_1\qquad e i_r \equiv j_r \pmod{\ell\!-\!1}, \qquad
\tau_e \equiv q^{j_1} + q^{j_2} \pmod{\ell}\notag
\end{align}
Finally, since $\tau_{e}$ is ostensibly the trace of the $e$th power of Frobenius at $\pp$, it is completely determined by $\tau_{1}$, and may be recovered by the Girard-Newton formulae for symmetric polynomials.

We now proceed by an exhaustive search. By Corollary \ref{cor:bal_unif_bound}, we have $\ell < 1.7 \times 10^7$, and since $e$ is bounded (Corollary \ref{cor:g_1_e_bound}), there are only finitely many choices of $\ell$ and $(j_1, j_2)$ to consider.  For each such triple $(\ell,j_1,j_2)$, at each prime $\mf{p}$ of $K$ there are only finitely many candidate values for each of $f,i_1,i_2,\tau_1$.  A computational search in \Sage \cite{SageMath} demonstrates that for every such triple $(\ell,j_1,j_2)$, there exists a prime $p\leq 11$ for which the system \eqref{eqn:tau_constraints} has no solutions, except for the six exceptions identified in the statement of the proposition. The code used for this calculation is publicly available at \cite{McLeman-Rasmussen:2026}. 
\end{proof}

\section{Comparison of trace behavior}\label{sec:frob_trace_heavenly}

In this and the following section, we characterize, among elliptic curves with complex multiplication defined over a quadratic field, those curves which are balanced at some prime $\ell \neq 2$. Suppose $E/K$ is balanced at a prime $\mf{l} \mid \ell$, while $E'/K$ is a curve with complex multiplication and good reduction away from $\ell$. Let $\rho_{E,\ell}$ and $\rho_{E',\ell}$ be the associated Galois representations on $\ell$-torsion. We show that the traces of Frobenius elements under $\rho_{E,\ell}$ and under $\rho_{E',\ell}$ have similar behavior (compare Proposition \ref{prop:ap_mod_ell_for_CM} and Corollary \ref{cor:balanced_ap_mod_ell}). In addition to providing striking evidence for Conjecture \ref{conj:balanced_implies_cm}, this leads to Theorem \ref{thm:heavenly_iff_nonsurj_trace}, a necessary and sufficient condition for a curve with complex multiplication to be heavenly and balanced. Many results in \S6 and \S7 rely on Theorem \ref{thm:GenusTheory}. Consequently, the condition $j(E) \not\in \QQ$ is often present.

Let us first recall the classical theory when complex multiplication is present. Suppose $E/K$ has complex multiplication by an imaginary quadratic field $L$ and good reduction away from one rational prime $\ell \neq 2$. By applying genus theory, we can better understand the relationship between $K$, $L$, and $\ell$.

\begin{theorem}\label{thm:GenusTheory}
Let $m \geq 1$ be squarefree and set $L = \QQ(\sqrt{-m})$. Suppose $E$ is an elliptic curve with complex multiplication by $\O_L$, defined over the quadratic field $K=\QQ(j(E))$. Suppose further that $E$ has good reduction away from an odd prime $\ell$. Let $M := KL = L(j(E))$.  Then
\begin{enumerate}[nosep, left=1ex, label={(\alph*)}]
    \item $K=\QQ(\sqrt{p})$ for some prime $p \not\equiv 3\bmod{4}$,
    \item $M=\QQ(\sqrt{p}, \sqrt{-\ell})$, and
    \item $\ell \equiv 3 \pmod{4}$ and $m=\ell p$ (so $\ell\mid m$). \qedhere
\end{enumerate}
\end{theorem}

\begin{proof}
Since $E$ has complex multiplication, $j(E) \in \RR$ \cite[\S10.3 Remark 1]{Lang:1987} and $K$ is a real quadratic field. Because $E$ has good reduction away from $\ell$, the extension $K(E[\ell])/K$ is unramified outside $\ell$. Further, by \cite[Lemma 3.15]{BCS:2017}, $L \subseteq K(E[\ell])$ and so also $M \subseteq K(E[\ell])$; thus $M/K$ is unramified at finite primes away from $\ell$.  

By \cite[Thm.~2.2]{Silverman:1994}, $M = L(j(E))$ is the Hilbert class field of $L$, so $L$ has class number $2$.  By genus theory the discriminant $\Delta_L$ of $L$ has exactly two prime factors \cite[Thm.~6.1]{Cox:Primes}, and further, $M$ is the genus field of $L$, known to satisfy $M=L(\sqrt{p})$ for any prime $p\equiv 1\bmod{4}$ dividing $\Delta_L$ or $M=L(\sqrt{-q})$ for any prime $q\equiv 3\bmod{4}$ dividing $\Delta_L$. 
Since $\Delta_L$ has exactly two prime divisors, one of the following cases must hold:
\begin{enumerate}[nosep, left=0pt, label={(\roman*)}]
\item $m=pq$ with $p\not\equiv 3\pmod{4}$ and $q\equiv 3\pmod{4}$, yielding $M = \QQ(\sqrt{p},\sqrt{-q})$, or
\item $m=cp$ with $p\equiv 1\pmod{4}$ and $c\in\{1,2\}$, yielding $M = \QQ(\sqrt{-cp},\sqrt{p})$.
\end{enumerate}
In either case, the only real quadratic subfield of $M$ is $\QQ(\sqrt{p})$, proving (a).  In case (ii), both $K/\QQ$ and $M/K$ are unramified at $2$, but $\QQ(\sqrt{-c}) \subseteq M$, so $M/\QQ$ is ramified at $2$, a contradiction. Thus case (i) must hold and now $q$ evidently ramifies in $M/K$. We conclude $q = \ell$, proving (b) and (c).
\end{proof}

\begin{corollary}\label{cor:no_EGR}
An elliptic curve $E/K$ over a quadratic field $K = \QQ(j(E))$ with complex multiplication by $\O_L$ cannot have everywhere good reduction.
\end{corollary}
\begin{proof}
Pick any prime $\ell' \equiv 1 \bmod{4}$. If $E/K$ has everywhere good reduction, then Theorem \ref{thm:GenusTheory} applies with $\ell = \ell'$, forcing $\ell' \equiv 3 \bmod{4}$, a contradiction. 
\end{proof}

\begin{corollary}\label{cor:two_primes}
An elliptic curve $E/K$ over a quadratic field $K = \QQ(j(E))$ with complex multiplication by $\O_L$ cannot be heavenly at two distinct primes.
\end{corollary}
\begin{proof}
Such a curve has everywhere good reduction and Corollary \ref{cor:no_EGR} applies.
\end{proof}

Recall that a curve $E'/K$ with complex multiplication by a nonmaximal order in a field $L$ is isogenous to a curve $E/K$ with complex multiplication by the maximal order $\O_L$ (see Remark \ref{rmk:heavenly_isog}).  Since the reduction type of a curve is isogeny-invariant, Corollaries \ref{cor:no_EGR} and \ref{cor:two_primes} also hold for $E'$, provided $j(E)\notin\Q$. This leaves the case where $j(E')\notin\Q$ but $j(E)\in\Q$, which does occur. For example, set $L = \QQ(\sqrt{-2})$ and $K = \QQ(a)$ where $a = \sqrt{6}$. Consider the following curves, both defined and isogenous over $K$:
\begin{align*}
E_1 \colon y^2 + axy + (a + 1)y & = x^{3} + (a+1) x^{2} + (17a - 41)x - (57a - 138), \\
E_2 \colon y^2 + axy + (a + 1)y & = x^{3} + (a + 1) x^{2} + (2a - 1)x + (a - 1).
\end{align*}
Both curves have everywhere good reduction, and complex multiplication by an order in $L$. For $E_1$, the order is $\OO = \ZZ + 3\OO_L$ while for $E_2$ it is the maximal order $\OO_L$. In fact, these curves are heavenly at both $\ell = 2$ and $\ell = 3$, but no contradiction with Corollary \ref{cor:no_EGR} or Corollary \ref{cor:two_primes} exists: the curve $E_1$ does not have complex multiplication by $\OO_L$, while $E_2$ has $j(E_2) = -8000$, so $K \neq \QQ(j(E_2))$. Curves in the $K$-isogeny class of $E_1$ (class $H(6.24).1$ in Table \ref{table:all_non_rational_j_curves}) are heavenly at both $\ell = 2$ and $\ell = 3$; to the authors' knowledge, these are the first known examples of abelian varieties which are heavenly at more than one prime $\ell$. 

Resume the notation and hypotheses of Theorem \ref{thm:GenusTheory}. We consider the behavior of the Hecke character attached to $E/K$ (see \cite[Ch.~II]{Silverman:1994} or \cite{Rubin-Silverberg:2009} for details). Set $E_M := E \times_K M$ and suppose $p \neq \ell$. Let $\mf{Q}$ be a prime of $M$ above $p$ and let $\mf{p} = \mf{Q} \cap K$, $\mf{P} = \mf{Q} \cap L$ be the corresponding primes of $K$ and $L$, respectively. Since $E_M$ has good reduction at $\mf{Q}$, the Hecke character attached to $E_M/M$ produces an algebraic integer $\xi \in \O_L$ which generates the (necessarily principal) ideal $\mathbf{N}_{M/L}\mf{Q}$. It is well known that the trace of $\xi$ is $a_{\mf{Q}}(E_M)$ and the norm of $\xi$ is $\mathbf{N}_{L/\QQ}(\xi) = p^{f_{\mf{Q}}}$. Consequently, if $\xi = (u + v\sqrt{-m})/2$ with $u, v \in \ZZ$, then
\begin{equation}\label{eqn:hecke_char_identity}
4p^{f_{\mf{Q}}} = a_{\mf{Q}}(E_M)^2 + mv^2.
\end{equation}
\begin{proposition}\label{prop:ap_mod_ell_for_CM}
Let $K$ be a quadratic field, $\ell$ an odd prime, and $\mf{p}$ a prime of $K$ above $p$ not dividing $\ell$, with residue field degree $f = f_{\mf{p}}$.  Let $L=\QQ(\sqrt{-m})$, and suppose $E/K$ is an elliptic curve with complex multiplication by $\O_L$ and good reduction away from $\ell$.
\begin{enumerate}[label={(\alph*)}, nosep, left=0pt]
\item If $\mf{p}$ is inert in $KL$, then $a_{\pp}(E) =0$.
\item If $\mf{p}$ splits in $KL$, then $a_\mf{p}(E)^2\equiv 4p^f\bmod{\ell}$.
\end{enumerate}
\end{proposition}
\begin{proof}
Part (a) is \cite[Ex.~II.2.30(c)]{Silverman:AEC}. For (b), let $M=KL$ and $\mf{Q}$ be a prime of $M$ above $\mf{p}$. Let $p$ and $\mf{P}$ be the primes of $\QQ$ and $L$ below $\mf{Q}$, respectively. Since $\mf{p}$ splits in $KL$, the residue fields of $\mf{p}$ and $\mf{Q}$ coincide, giving $f=f_\mf{Q}$ and $a_\mf{p}(E)=a_{\mf{Q}}(E_M)$.  Consequently, \eqref{eqn:hecke_char_identity} implies $4p^f = a_\mf{p}(E)^2 + m v^2$. Since $\ell\mid m$ by Theorem \ref{thm:GenusTheory}, the second claim follows.
\end{proof}

\begin{remark}\label{rmk:ConverseForCM}
We note that a converse to Proposition \ref{prop:ap_mod_ell_for_CM} exists. Suppose $K$ is a real quadratic field, $E/K$ an elliptic curve with good reduction away from $\ell$, and let $M/K$ be any quadratic extension. 
It is known \cite[Theorem 20, Remark 2]{Serre:1981} that if the set of primes $\pp$ for which $a_{\pp}(E) = 0$ has positive density, then $E/K$ must have complex multiplication. So if the primes of $K$ which are inert in $M/K$ all satisfy $a_{\pp}(E) = 0$, it follows that $E$ has complex multiplication.  
\end{remark}

Now let us change perspective. Let $K$ be a quadratic field and assume $E/K$ is an elliptic curve which is balanced at $\mf{l}$, a prime of $K$ above the rational odd prime $\ell$. To be clear, we make no assumption that $E$ has complex multiplication.


\begin{lemma}\label{lemma:ell_mod_d}
Suppose $\ell>3$ and $\mf{l}$ a prime of $K$ above $\ell$ with $f_{\mf{l}} = 1$.  Suppose $E/K$ is an elliptic curve balanced at $\mf{l}$.  Then for any divisor $d$ of $e(E/K;\mf{l})$ with $n_d = 1$, we have $\ell \not\equiv 1 \pmod{d}$.
\end{lemma}

The proof of Lemma \ref{lemma:ell_mod_d} relies on the following result, which is a special case of \cite[Prop.~6.7]{Rasmussen-Tamagawa:2017}. Let $K_{\mathfrak{l}}^{\ssimp}$ be as defined in \S\ref{sec:totid}, and set $E^{\ssimp} := E \times_{K} K_{\mathfrak{l}}^{\ssimp}$. Let $\mathcal{E}^{\ssimp}$ be the N\'{e}ron model of $E^{\ssimp}$ over $\mathcal{O}_{K^{\ssimp}_{\mathfrak{l}}}$.

\begin{lemma}\label{lemma:RT67-as-we-need-it}
    Suppose that $\ell \nmid e$ and that $E^{\ssimp}$ has good reduction with $\ell$-rank zero. Suppose $\mathfrak{l}$ is a prime of $K$ above $\ell$ and $f_{\mathfrak{l}} = 1$. If $d \mid e(E/K;\mathfrak{l})$ and $n_{d} = 1$, then $\ell \not\equiv 1 \pmod{d}$. 
\end{lemma}

\begin{remark}
    For clarity, we explain how Lemma \ref{lemma:RT67-as-we-need-it} follows from \cite[Prop.~6.7]{Rasmussen-Tamagawa:2017}. The cited Proposition shows that, provided two hypotheses (A3) and (A4) are satisfied, various implications hold among conditions labeled (C1) through (C8). In the current context, the good reduction assumption gives (A3), and the assumption $\ell \nmid e$ is exactly (A4). The hypothesis of $\ell$-rank zero implies condition (C4); (C4) implies (C8) by the Proposition itself; in turn (C8) implies $\ell \not\equiv 1 \pmod{d}$ by the remark following the statement of the Proposition.
    \end{remark}

\begin{proof}[Proof of Lemma \ref{lemma:ell_mod_d}.]
    Since $\ell > 3$, we certainly have $\ell \nmid e$. Because $E/K$ is balanced, the Tate-Oort numbers satisfy $j_{1} = j_{2} = \frac{e}{2}$; in particular, $j_{r} < e$ for both $r=1, 2$. The curve $E^{\ssimp}$ has semistable reduction, which must be either good or multiplicative. Let us consider the case of multiplicative reduction. First, since the residue field for $K_{\mathfrak{l}}^{\ssimp}$ is algebraically closed, the multiplicative reduction is necessarily split. The special fiber of $\mathcal{E}^{\ssimp}$ has $\mathbb{G}_{m}$ as its identity component. Now the $\ell$-torsion will have, as its connected part, the group scheme $\bmu_{\ell}$ (the $\ell$-torsion of $\mathbb{G}_{m}$). Since $\psi_{\mathfrak{l}}^{e} = \chi$ on $J_{\mathfrak{l}}$, we conclude the associated Tate-Oort number for this piece is $e$, in contradiction with the balanced assumption.
    
Consequently, $E^{\ssimp}$ has good reduction. The special fiber of $\mathcal{E}^{\ssimp}$ is an elliptic curve over the residue field, and must be either ordinary ($\ell$-rank $1$) or supersingular ($\ell$-rank $0$). In the ordinary case, consider the connected-\'{e}tale sequence for $\mathcal{E}^{\ssimp}[\ell]$ as a finite flat group scheme over $\mathcal{O}_{K_{\mathfrak{l}}^{\ssimp}}$:
    \[ 0 \longrightarrow \mathcal{E}^{\ssimp}[\ell]^{\circ} \longrightarrow \mathcal{E}^{\ssimp}[\ell] \longrightarrow \mathcal{E}^{\ssimp}[\ell]^{\text{\textsf{\'{e}t}}} \longrightarrow 0, \qquad \mathcal{E}^{\ssimp}[\ell]^{\circ} \cong \bmu_{\ell}. \]
    The presence of the $\bmu_{\ell}$ component again violates the balanced assumption. Thus, the special fiber of $\mathcal{E}^{\ssimp}$ is supersingular, i.e., has $\ell$-rank zero. The result now holds by Lemma \ref{lemma:RT67-as-we-need-it}.
\end{proof}

\begin{proposition}\label{prop:ell_is_3_mod_4}
Suppose $\ell > 3$, and suppose $E/K$ is balanced at a prime $\mf{l}$ of $K$ above $\ell$. Then $\ell \equiv 3 \bmod{4}$. Moreover, if $e_E(\mf{l}) = 12$, then $\ell \equiv 11 \bmod{12}$. 
\end{proposition}
\begin{proof}
Set $e = e_E(\mf{l})$. Note $e_{\mf{l}} \leq [K:\QQ] = 2$, and that by Proposition \ref{cor:g_1_e_bound}, $e(E/K; \mf{l}) \leq 6$. Thus $e = e(E/K; \mf{l}) \cdot e_{\mf{l}} \leq 12$. By Lemma \ref{lemma:4_divides_e}, $e \in \{4, 8, 12 \}$; further, if $e \neq 8$ then $\ord_2(\ell-1) < \ord_2 e = 2$, and so $\ell \equiv 3 \bmod{4}$. If $e \neq 4$, write $e = 2h$ with $h \in \{4, 6 \}$. Since $e > 6$ we see $e_{\mf{l}} = 2$ and $f_{\mf{l}} = 1$. The totient identity \eqref{eqn:tot_id} forces $n_h = 1$, and by Lemma \ref{lemma:ell_mod_d} we have $\ell \not\equiv 1 \bmod{h}$. If $e = 8$, this implies $\ell \equiv 3 \bmod{4}$. If $e = 12$, we have both $\ell \equiv 3 \bmod{4}$ and $\ell \not\equiv 1 \bmod{6}$, hence $\ell \equiv 11 \bmod{12}$.
\end{proof}

\begin{lemma}\label{lemma:bal_gives_i1_i2}
Suppose $E/K$ is balanced at a prime $\mf{l}$ of $K$ above $\ell \neq 2$. The powers $i_1, i_2$ of $\chi$ appearing on the diagonal of $\rho_{E,\ell}$ are $\{i_1, i_2\} = \{(\ell + 1)/4, (3\ell - 1)/4 \}$.
\end{lemma}
\begin{proof}
Note that as $\det \rho_{E,\ell} = \chi$, we have $i_2 \equiv 1 - i_1 \bmod{(\ell-1)}$. If $\ell = 3$ this implies
\[ \{i_1, i_2 \} \equiv \{0, 1 \} \equiv \{(\ell+1)/4, (3\ell - 1)/4 \} \bmod{(\ell-1)}. \]
So for the remainder we assume $\ell > 3$. Since $E$ is balanced, the Tate-Oort numbers of $E$ at $\mf{l}$ are $j_1 = j_2 = \frac{e}{2}$. Now \eqref{eqn:i_j_reln} implies 
\begin{equation}\label{eqn:gcd_ir_2}
2e i_r \equiv e \bmod{(\ell-1)}
\end{equation}
for each $r$. As $\ell \equiv 3 \bmod{4}$, $\gcd(2e, \ell - 1) = 2$ and so there are only two solutions to \eqref{eqn:gcd_ir_2} satisfying $0 \leq i_r < \ell - 1$, namely $(\ell + 1)/4$ and $(3\ell - 1)/4$.
\end{proof}

\begin{theorem}\label{thm:trace_values}
Suppose $E/K$ is balanced at a prime of $K$ above $\ell$, and $\mf{p}$ is a prime of $K$ above $p \neq \ell$, with residue field degree $f = f_{\pp}$ and norm $q = \mathbf{N}\pp = p^f$. Then 
\begin{equation}\label{eqn:trace_values}
a_{\mf{p}}(E) \equiv q^{ (\ell+1)/4 } \left( 1 + \left( \tfrac{p}{\ell} \right)^{\smash{f}} \right) \pmod{\ell}.
\end{equation}
\end{theorem}
\begin{proof}
Let $\theta_{\mf{p}} \in G_K$ be a Frobenius element for $\mf{p}$. As $\chi(\theta_\mf{p}) = q$ and $\{i_1, i_2 \} = \{(\ell + 1)/4, (3\ell - 1)/4 \}$,
\[ a_{\mf{p}}(E) \equiv \tr \rho_{E,\ell}(\theta_{\mf{p}}) \equiv q^{i_1} + q^{i_2} \equiv q^{(\ell + 1)/4} \left(1 + ( p^{(\ell-1)/2} )^f \right) \pmod{\ell}, \]
and now the result follows from Euler's criterion.   
\end{proof}
We adopt the convention that any prime in a trivial extension splits. So in the following corollary, if $K = KL$, case (b) holds.
\begin{corollary}\label{cor:balanced_ap_mod_ell}\label{Cor:MainIntroResult}
Assume the hypotheses of Theorem \ref{thm:trace_values}. Set $L = \QQ(\sqrt{-\ell})$.
\begin{enumerate}[label={(\alph*)}, nosep, left=0pt]
\item If $\mf{p}$ is inert in $KL$, then $a_{\pp}(E) \equiv 0 \bmod{\ell}$.
\item If $\mf{p}$ splits in $KL$, then $a_\mf{p}(E)^2\equiv 4p^f \bmod{\ell}$.
\end{enumerate}
\end{corollary}
\begin{proof}
As $\ell\equiv 3\pmod{4}$ we have $\bigl( \frac{-\ell}{p} \bigr) = \left( \frac{p}{\ell} \right)$. Thus $p$ splits in $L$ if and only if $\bigl(\frac{p}{\ell}) = 1$. Let $\mf{P}$ be a prime of $KL$ above $\pp$. For (a), suppose $\mf{p}$ is inert in $KL$, i.e. $f_{\mf{P}|\mf{p}} = 2$. Then $f_{\mf{P}} = f \cdot f_{\mf{P}|\mf{p}} \geq 2$. Since $KL/\QQ$ is either quadratic or biquadratic, the (necessarily cyclic) decomposition group for $\mf{P}$ cannot be order $4$. So in fact $f_{\mf{P}} = 2$ and $f = 1$. Consequently, $p$ is inert in $L$ and $(\frac{p}{\ell}) = -1$. By Theorem \ref{thm:trace_values}, $a_{\mf{p}}(E) \equiv 0 \bmod{\ell}$. For (b), suppose $\mf{p}$ splits in $KL$, i.e., $f_{\mf{P}|\mf{p}} = 1$. If $f = 1$ then $f_{\mf{P}} = 1$; this means $p$ splits completely in $KL$ and in particular $p$ splits in $L$. Thus $(\frac{p}{\ell}) = 1$. Since either $f = 2$ or $(\frac{p}{\ell}) = 1$, we may be sure $(\frac{p}{\ell})^f = 1$. Theorem \ref{thm:trace_values} now implies
\[
a_\mf{p}(E)^2 \equiv 4q^{ (\ell + 1)/2 } \equiv 4 \bigl( p^{(\ell + 1)/2 }\bigr)^f \equiv 4\left( p \cdot (\tfrac{p}{\ell}) \right)^f \equiv 4 p^f \left( \tfrac{p}{\ell} \right)^f \equiv 4p^f\pmod{\ell}. \qedhere
\]
\end{proof}

\begin{remark}\label{rmk:heavenly_converse}
If one could improve Corollary \ref{cor:balanced_ap_mod_ell} to obtain $a_{\pp}(E) = 0$ for inert $\pp$, then the argument outlined in Remark \ref{rmk:ConverseForCM} would apply, and Conjecture \ref{conj:balanced_implies_cm} would follow. This is the evidence alluded to at the start of the section: a counterexample to Conjecture \ref{conj:balanced_implies_cm} would be an elliptic curve $E$ without complex multiplication (and so satisfying the Sato-Tate distribution), while also satisfying $\ell \mid a_{\pp}(E)$ for every inert $\pp$ in $KL$.
\end{remark}

\begin{remark}\label{rmk:Frey-Mazur}
Let $\ell$ be a rational prime, and suppose that $E/K$ and $E'/K$ are elliptic curves which are balanced at primes $\mf{l}$, $\mf{l}'$ of $K$ above $\ell$. Theorem \ref{thm:trace_values} implies $a_{\mf{p}}(E) \equiv a_{\mf{p}}(E') \bmod{\ell}$ for all but finitely many primes $\mf{p}$ of $K$, and hence an isomorphism of Galois representations, $\rho_{E,\ell} \cong \rho_{E',\ell}$. A natural extension of the Frey-Mazur conjecture (originally a question asked by Mazur \cite[pg.~133]{Mazur:78}) over quadratic fields would imply the following: for sufficiently large $\ell$, any two elliptic curves over $K$ which are heavenly at $\ell$ must be $K$-isogenous. 
\end{remark}

The next result demonstrates that $\trace \rho_{E,\ell}$ is not surjective in the case of balanced curves. This is analogous to \cite[Lemma 7.1]{David:1999}, a ``missing traces'' result on certain elliptic curves $E/\QQ$ with complex multiplication.

\begin{proposition}\label{prop:square_traces}
Suppose $E/K$ is an elliptic curve, balanced at a prime $\mf{l}$ of $K$ above $\ell$. Then
\[\trace \rho_{E,\ell}(G_K) = \left( \tfrac{2}{\ell} \right) \cdot \FF_{\ell}^{\times 2} \cup \{0 \}. \qedhere \]
\end{proposition}

\begin{proof}
To show $\trace \rho_{E,\ell}(G_K) \subseteq (\frac{2}{\ell}) \FF_{\ell}^{\times 2} \cup \{0 \}$, it suffices to show $a_{\mf{p}} \equiv 0 \bmod{\ell}$ or $( \frac{a_{\mf{p}}}{\ell} ) = ( \frac{2}{\ell} )$ for any prime $\mf{p} \nmid \ell$. If $\left( \frac{p}{\ell} \right) = 1$, then Theorem \ref{thm:trace_values} implies $a_{\mf{p}} \equiv 2 p^{f(\ell+1)/4} \bmod{\ell}$, and so $( \frac{a_{\mf{p}}}{\ell}) = (\frac{2}{\ell})$, i.e., $a_{\mf{p}} \bmod{\ell} \in (\frac{2}{\ell}) \FF_{\ell}^{\times 2}$. If $(\frac{p}{\ell}) = -1$, then either $a_{\pp} \equiv 0 \bmod{\ell}$ or $f = 2$. In the latter case, note that $\ell \equiv 3 \bmod{4}$ since $E$ is balanced at $\mf{l}$. Thus, $\frac{\ell + 1}{2}$ is even and $q^{(\ell + 1)/4} = p^{(\ell + 1)/2}$ is a square modulo $\ell$. Now Theorem \ref{thm:trace_values} implies $(\frac{a_{\mf{p}}}{\ell}) = (\frac{2}{\ell})$.

For the reverse containment, suppose $a \in (\frac{2}{\ell}) \FF_{\ell}^{\times 2} \cup \{0 \}$. We must demonstrate $\pp$ such that $a_{\mf{p}} \equiv a \bmod{\ell}$. If $a = 0$ then $a \equiv a_{\mf{p}}$ for any prime $\mf{p}$ above a rational prime $p$ which is inert in $\QQ(\sqrt{-\ell})$. For $a \neq 0$, set $b = \frac{a}{2} \in \FF_{\ell}^\times$ and choose a rational prime $p$ which splits in $K$ and satisfies $p \equiv b^2 \bmod{\ell}$. Clearly $(\frac{p}{\ell}) = 1$ and so $p$ splits in $\Q(\sqrt{-\ell})$. We have $f = 1$ and by Corollary \ref{Cor:MainIntroResult}, $a_{\mf{p}}^2 \equiv a^2 \bmod{\ell}$, i.e., $a_{\mf{p}} \equiv \pm a \bmod{\ell}$. By the first part of the proof we know $a_{\pp} \in (\frac{2}{\ell}) \FF_{\ell}^{\times 2}$. Since $\ell \equiv 3 \bmod{4}$, $-1$ is a quadratic nonresidue modulo $\ell$. Thus $-a \not\in (\frac{2}{\ell}) \FF_{\ell}^{\times 2}$ and $a_{\mf{p}} \equiv a \bmod{\ell}$.
\end{proof}

\section{Heavenly elliptic curves with complex multiplication.}\label{sec:heavenly_with_cm}

In the case of heavenly curves with complex multiplication, we are able to both characterize heavenliness in terms of the trace of Frobenius, and also give a complete classification of such curves. 

\begin{theorem}\label{thm:heavenly_iff_nonsurj_trace}
Let $\ell \neq 2$ be prime and $K$ a quadratic field. Suppose that $E/K$ is an elliptic curve that satisfies $K = \QQ(j(E))$. Further, assume that $E/K$ has complex multiplication and has good reduction away from $\ell$.   Then
\begin{enumerate}[nosep, label={(\alph*)}, left=1ex]
\item  $E/K$ is heavenly at $\ell$ if and only if $\trace(\rho_{E,\ell}(G_K)) \neq \FF_\ell$. 
\item If $E/K$ is heavenly at $\ell$ and $\ell > 7$, then $E$ is balanced at $\ell$.
\end{enumerate}
\end{theorem}


\begin{proof}
By Lemma \ref{lemma:heavenly_criteria}, $E/K$ is heavenly at $\ell$ if and only if $[K(E[\ell]):K(\bmu_\ell)]$ is a power of $\ell$. Let $L=\Q(\sqrt{-m})$ be the field of complex multiplication. By Remark \ref{rmk:heavenly_isog}, we may and do assume the complex multiplication is by the maximal order $\O_L$. Set $M = KL$. From \cite[Lemma 3.15]{BCS:2017} we know $L \subseteq K(E[\ell])$ and so $M(E[\ell]) = K(E[\ell])$. By Theorem \ref{thm:GenusTheory}, $K = \QQ(\sqrt{p})$ and $m=\ell p$ where $\ell \equiv 3 \bmod{4}$ and $p$ is a prime with $p \not\equiv 3 \bmod{4}$. Since $K$ is real, $[K(\bmu_\ell):K] = \ell - 1$, and since $M = K(\sqrt{-\ell}) \subseteq K(\bmu_\ell)$, it follows that $M(\bmu_\ell) = K(\bmu_\ell)$ and $[M(\bmu_\ell):M] = (\ell-1)/2$.

  Let $\rho = \rho_{E,\ell}$ be the $G_K$-representation on $E[\ell]$. Because $E$ has complex multiplication and  $\ell$ ramifies in $L$, we know $E[\ell]$ is a free rank one module over  $\O_L/\ell \O_L \cong \FF_\ell[T]/(T^2)$. Consequently the image of $G_M$ under $\rho$ is abelian, lying inside 
  \[ \Aut_{\O_L/\ell \O_L}(E[\ell]) \cong (\FF_\ell[T]/(T^2))^\times = \left\{ bT + a + (T^2) : b \in \FF_\ell, a \in \FF_\ell^\times \right\}. \]
(For a detailed explanation, see \cite[\S II.1.4, II.2.3]{Silverman:1994}.) By choosing an ordered basis for $E[\ell]$ that corresponds to the $\FF_\ell$-basis $\{T, 1 \}$ of $\FF_\ell[T]/(T^2)$, we find
  \begin{equation}\label{eqn:GM_repn}
 \rho(G_M) \leq \mathcal{C} := \left\{ \begin{pmatrix*} a &b \\ 0 & a \end{pmatrix*} : b \in \FF_\ell,\ a \in \FF_\ell^\times \right\}.
 \end{equation}
 Let $\mathcal{N}$ be the normalizer of $\mathcal{C}$ inside $\GL_2(\FF_\ell)$. The quotient $\mathcal{N}/\mathcal{C}$ is cyclic of even order (in fact isomorphic to $\FF_\ell^\times$). Let $\mathcal{N}'$ be the unique subgroup of $\mathcal{N}$ satisfying $[\mathcal{N}' : \mathcal{C}] = 2$, namely
 \begin{equation}\label{eqn:GK_repn}
 \mathcal{N}' := \left\{ \begin{pmatrix*} \pm a & b \\ 0 & a \end{pmatrix*} : b \in \FF_\ell, a \in \FF_\ell^\times \right\}.
 \end{equation}
By \cite[Lemma 6.2]{Lozano-Robledo:2022}, $\rho(G_K) \leq \mathcal{N}'$. Let $\varepsilon \colon G_K \to \{ \pm 1 \}$ be the cokernel character for the natural projection $G_K \twoheadrightarrow G_K/G_M$. By \eqref{eqn:GM_repn},  there exists a homomorphism $\theta \colon G_K \to \FF_\ell^\times$ such that
\[ \rho = \begin{pmatrix*} \varepsilon \theta & \star \\ & \theta \end{pmatrix*}, \qquad {\textstyle \eval{\rho}_{G_M}} = \begin{pmatrix*} \psi & \star \\ & \psi \end{pmatrix*}, \qquad \psi := {\textstyle \eval{\theta}_{G_M}}, \]
and these necessarily satisfy $\varepsilon \theta^2 = \det \rho = \chi$. For any $\sigma \in G_K$ we have 
\[ \trace \rho(\sigma) = 0 \iff \varepsilon(\sigma) = -1 \iff \sigma \not\in G_M, \]
and so $\trace \rho(G_K) = \{0 \} \cup \trace \rho(G_M)$. But $2 \in \FF_\ell^\times$ and therefore $\sharp\!\left( \trace \rho(G_M) \right) = \sharp\!\left(2 \cdot \psi(G_M) \right) = \sharp \psi(G_M)$. Note $\trace \rho(G_K)= \FF_\ell$ if and only if $\psi$ is surjective. Let $M^\psi = M^{\ker \psi}$. Since $\psi^2 = \eval{\chi}_{G_M}$, $G_{M^\psi} = \ker \psi \leq \ker \psi^2 = \ker {\textstyle \eval{\chi}_{G_M}} = G_{M(\bmu_\ell)}$, i.e., $M(\bmu_\ell) \subseteq M^{\psi}$. Precisely by the definition of $M^{\psi}$, $\eval{\rho}_{M^{\psi}}$ is unit upper triangular and so $[M(E[\ell]) : M^{\psi}]$ divides $\ell$. On the other hand,
\[ [M^\psi : M(\bmu_\ell)] = \frac{ [M^\psi : M] }{ [M(\bmu_\ell):M] } = \frac{ \sharp \psi(G_M) }{ (\ell - 1)/2 } = \begin{cases} 2 & \psi\ \text{surjective}, \\ 1 & \psi\ \text{not surjective.} \end{cases} \]
Thus, $[K(E[\ell]):K(\bmu_\ell)] = \ell$ if and only if $M^{\psi} = M(\bmu_\ell)$, if and only if $\trace \rho(G_K) \neq \FF_\ell$, proving (a).

Now assume $\ell > 7$. Since $E$ is heavenly, we have two forms for the representation $\rho$:
\[ \rho \sim \begin{pmatrix*} \varepsilon \theta & \star \\ & \theta \end{pmatrix*} \sim \begin{pmatrix*} \chi^{i} & \star \\ & \chi^{1-i} \end{pmatrix*}. \]
It follows that $\varepsilon = \varepsilon \theta \theta^{-1} = \chi^{\pm (2i-1)}$.  Thus $4i - 2 \equiv 0 \bmod{(\ell-1)}$, i.e., $i \in \{ (\ell + 1)/4, (3\ell - 1)/4 \}$. In particular, $\ell \equiv 3 \bmod{4}$.  If $\ell > 19$ then $E$ is balanced at $\ell$ by Proposition \ref{prop:B21_is_19}. When $\ell = 19$ or $\ell = 11$, \eqref{eqn:i_j_reln} can only hold if $E$ is balanced at $\ell$ by a straightforward calculation.
\end{proof}

\begin{remark}
In case $\ell = 7$, \eqref{eqn:i_j_reln} does not quite imply $E$ must be balanced at $\ell$, but these congruences do imply that at any $\mf{l} \mid \ell$, either $E$ is balanced at $\mf{l}$ or else $3 \mid e$ and the Tate-Oort numbers are $(j_1, j_2) = (0, e)$.
\end{remark}

Among elliptic curves defined over a quadratic number field and possessing complex multiplication, we give a classification for those which are heavenly. 

\begin{theorem}\label{thm:heavenly_cm_calculation}
    Suppose $(K, [E]_{K}, \ell)$ is a triple where $\ell$ is a prime number, $K$ is a quadratic number field, $E/K$ is an elliptic curve with $K = \QQ(j(E))$. Suppose further that $E$ possesses complex multiplication and that $E$ is heavenly at $\ell$. Then
    \begin{enumerate}[label={(\alph*)}, nosep, left=0pt]
        \item $\ell \in \{2, 3, 7, 11, 23, 31, 47 \}$.
        \item $E$ is isomorphic, over $K$, to an elliptic curve appearing in Table \ref{table:all_non_rational_j_curves}.
    \end{enumerate}
  \end{theorem}

\begin{proof}
There exist only finitely many imaginary quadratic fields $L$ containing an order $\O$ of class number exactly $2$; such orders are indexed by a finite and explicitly known set of $j$-invariants \cite{Daniels:2015}. For each such $j$-invariant $j$, we set $K = \QQ(j)$ and find an integral model for a curve $A/K$ such that $j(A) = j$. 

Suppose $A$ is one such curve, and take $K = \QQ(j(A))$. By \cite[Thm.~10.10]{Lang:1987}, a prime of $K$ where $A$ has good reduction must be unramified in $KL/K$. Set
\[ \operatorname{Ram}_\Q(KL/K) = \{p \in \ZZ \colon \text{there exists }\pp \in \O_K \text{ with }\pp\mid p\text{ and }\pp\text{ ramified in }KL/K\}.
\]
Necessarily, $A$ has bad reduction at every $p \in \operatorname{Ram}_{\QQ}(KL/K)$, and the same is true for any quadratic twist of $A$ over $K$. It follows that if $\sharp\operatorname{Ram}_{\QQ}(KL/K) > 1$, there cannot exist a twist of $A$ which is heavenly. If $\operatorname{Ram}_{\QQ}(KL/K) = \{\ell\}$, then a twist of $A$ can be heavenly only at the prime $\ell$. If $\operatorname{Ram}_{\QQ}(KL/K)$ is empty, then a priori a twist of $A$ could be heavenly at any prime $\ell$; however, a curve heavenly at $\ell$ necessarily possesses a $K$-rational $\ell$-isogeny. We determine all possible prime degrees for a $K$-isogeny of $A$; because $K$ is not contained in the CM field, there are only finitely many such $\ell$ for each $A$. This produces a finite set $\mathcal{A}$ of pairs $(A,\ell)$ with the following property: for any $(K, [E]_{K}, \ell)$ satisfying the hypotheses of Theorem \ref{thm:heavenly_cm_calculation}, $E$ is a quadratic twist (over $K$) of $A$, where $(A,\ell) \in \mathcal{A}$.
 
For each pair $(A,\ell) \in \mathcal{A}$, we search for quadratic extensions $F/K$ such that the quadratic twist $A^{F}$ of $A$ determined by $F/K$ has good reduction away from $\ell$. There are only finitely many $F/K$ to consider, because the ramification of $F/K$ is limited to the primes of bad reduction of $A$ and the primes which divide $2\ell$. Let $S := \{\pp \subseteq \OK : \pp \mid 2\ell \Delta_{A} \}$, where $\Delta_{A}$ is the discriminant of $A$. By Kummer theory, $F$ must take the form $F = K(\sqrt{u})$, where $u \in \O_{K,S}^{\times}$. Thus, each possible field $F$ corresponds to one class in the finite set $\O_{K,S}^{\times}/\O_{K,S}^{\times 2}$. This determines a finite set $\mathcal{T}$ of pairs $(A,\ell)$ with the following property: for any triple $(K, [E]_{K}, \ell)$ meeting the hypotheses of Theorem \ref{thm:heavenly_cm_calculation}, there exists $(A,\ell) \in \mathcal{T}$ such that $E \cong_{K} A$. 

By design, for each pair $(A,\ell) \in \mathcal{T}$, we know $A$ has good reduction away from $\ell$. The curve $A$ will be heavenly at $\ell$ if and only if there exists a nontrivial $\ell$-torsion point $P \in E(K(\bmu_{\ell}))$. The existence of such a $P$ implies 
\[ \rho_{A/K(\bmu_\ell), \ell} \sim \begin{pmatrix*} 1 & \star \\ 0 & 1 \end{pmatrix*}, \]
and so $A$ is heavenly at $\ell$ by Lemma \ref{lemma:heavenly_criteria}. This calculation was carried out for each pair $(A,\ell) \in \mathcal{T}$ to determine the heavenly curves listed in Table \ref{table:all_non_rational_j_curves}.
\end{proof}

We wrote a combination of routines in \texttt{SageMath} and \texttt{MAGMA} to carry out the calculations described in the proof, with the results given in Table \ref{table:all_non_rational_j_curves}. This code is available at \cite{McLeman-Rasmussen:2026}.

\section{Balanced curves without complex multiplication}\label{sec:non_cm}

Let $\ell \neq 2$ be a rational prime, and suppose $K$ is a quadratic field. In this section, we demonstrate necessary conditions for an elliptic curve $E/K$ to be heavenly and balanced at $\ell$; we do not assume $E$ has complex multiplication.  
 
We set the following notation. If $N > 0$ is relatively prime to $\ell$ with $(\frac{N}{\ell}) = 1$, we let $r_\ell(N)$ denote the least positive integer lift of a root of $X^2 - N \in \FF_\ell[X]$. For any $\ell$, we define 
\[ \mathcal{N}(\ell) := \{ p\ \text{prime} : (\tfrac{p}{\ell}) = 1\ \text{and}\ r_{\ell}(4p) > 2\sqrt{p} \}. \]
The set $\mathcal{N}(\ell)$ is always finite since $p > \frac{\ell^2}{16}$ guarantees $r_\ell(4p) \leq \frac{\ell}{2} \leq 2\sqrt{p}$.

\begin{lemma}\label{lemma:vertical_scarcity}
Suppose $E/K$ is heavenly and balanced at $\ell$. If $p$ is a rational prime with $p \in \mathcal{N}(\ell)$, then $p$ does not split in $K$.
\end{lemma}

\begin{proof}
We argue by contradiction. Suppose $p \in \mathcal{N}(\ell)$ splits in $K$. Let $\mf{p}$ be a prime of $K$ over $p$. Set $L := \QQ(\sqrt{-\ell})$ and note that as $(\frac{p}{\ell}) = 1$, it follows that $p$ splits in $L$. As a prime splits completely in a compositum of extensions if and only if it splits completely in each of the extensions, we conclude $p$ splits completely in $KL$, and hence $\mf{p}$ splits in $KL$. By Corollary \ref{Cor:MainIntroResult}, we conclude $a_{\mf{p}}^2 \equiv 4p \bmod{\ell}$. On the other hand, $p \neq \ell$ and so from the Weil bound we have $0 < |a_{\mf{p}}| \leq 2\sqrt{p} < r_{\ell}(4p)$, contradicting the minimality of $r_{\ell}(4p)$. 
\end{proof}

Lemma \ref{lemma:vertical_scarcity} shows that for any specific $\ell \gg 0$, it is difficult for a quadratic field $K$ to admit an elliptic curve $E/K$ which is heavenly at $\ell$. In a similar vein, for any particular $K$, we may conclude $\mathcal{H}(K,1,\ell) = \varnothing$ for many $\ell$.

\begin{lemma}\label{lemma:horizontal_scarcity}
  Let $K$ be a quadratic field, and suppose $p$ is a rational prime which splits in $K$. If $\ell > 4p$ and $(\frac{p}{\ell}) = 1$, then there does not exist an elliptic curve $E/K$ which is heavenly at $\ell$.
\end{lemma}

\begin{proof}
  Let $p$ be such a prime and let $\pp$ be a prime of $K$ above $p$. Suppose $E/K$ is heavenly at $\ell$, with $\ell > 4p$. By Corollary \ref{Cor:MainIntroResult}, $a_{\pp}^2 \equiv 4p \bmod{\ell}$. On the other hand, $\sqrt{p} \not\in \ZZ$ and $(4p, \ell) = 1$, and so the Weil bound implies the strict inequality $0 < a_{\pp}^2 < 4p$. Thus $0 < 4p - a_{\pp}^2 < 4p < \ell$, contradicting that $\ell\mid 4p - a_{\pp}^2$.
\end{proof}

If there exists $A/\QQ$ with $([A]_{\QQ}, \ell) \in \mathcal{H}(\QQ, 1)$, then $([A \times_{\QQ} K]_K, \ell) \in \mathcal{H}(K, 1)$ for every $K$ and necessarily $\mathcal{N}(\ell) = \varnothing$. For other $\ell$, however, Lemma \ref{lemma:vertical_scarcity} is quite restrictive; we illustrate this with two examples.

\begin{itemize}[left=0pt]
\item For $\ell = 47$, we have $\mathcal{N}(47) = \{2, 3, 7, 17, 37, 53, 97 \}$. The primes in this set are inert in $K = \QQ(\sqrt{5})$, compatible with the existence of the heavenly isogeny class $H(47,5).1$ in Table \ref{table:all_non_rational_j_curves}. Ordered by magnitude of the discriminant, however, the next real quadratic field for which every prime in $\mathcal{N}(\ell)$ is inert is $K = \QQ(\sqrt{2309})$. By Lemma \ref{lemma:vertical_scarcity}, no quadratic field $K'$ with $5 < |\Delta_{K'}| < 2309$ admits an elliptic curve heavenly at $47$.
\item Note $\sharp\mathcal{N}(\ell) \to \infty$ as $\ell \to \infty$, and so this constraint is more severe for larger $\ell$.  For example, for $\ell = 103$, we have $\sharp\mathcal{N}(103) = 23$ and no quadratic field $K$ with $|\Delta_K| < 10^6$ has every prime in $\mathcal{N}(103)$ inert. Thus, Lemma \ref{lemma:vertical_scarcity} implies that none of these $K$ admit an elliptic curve $E/K$ which is heavenly at $103$.
\end{itemize}

Lemma \ref{lemma:horizontal_scarcity} is similarly restrictive. For instance, take $K = \QQ(\sqrt{7})$, and consider the set 
\[ S := \{ p : p\ \text{splits in}\ K,\ 2 < p < 250 \} = \{3, 19, 29, \dots, 233\}. \]
Applying Lemma \ref{lemma:horizontal_scarcity} to $p \in S$, we find $K$ cannot admit an elliptic curve heavenly at $\ell$ for any prime $\ell$ with $163 < \ell < 10^7$. In fact, we can say more. By (an effective version of) the Chebotarev Density Theorem, there exists $B_K$ such that for $\ell > B_K$ there exists a prime $p$ that satisfies the hypotheses of Lemma \ref{lemma:horizontal_scarcity}. Under assumption of the Extended Riemann Hypothesis, such bounds have been made explicit by Bach and Sorenson \cite{Bach-Sorenson:1996}, and are small enough to be of practical use. With $K = \QQ(\sqrt{7})$, the Bach-Sorenson bounds imply $B_K < 70000$. So, conditional on ERH, we deduce $\mathcal{H}(\QQ(\sqrt{7}), 1, \ell) = \varnothing$ for every $\ell > 163$. Similar ideas were used in \cite{Karpisz:2019} to prove, conditional on ERH, that $\mathcal{H}(K, 1, \ell) = \varnothing$ for any $\ell > 163$ and any quadratic field $K$ satisfying $|\Delta_K| < 100$. We extended this calculation and verified the following:
\begin{proposition}\label{prop:ERH_calculation}
    Suppose $\ell > 163$ is prime, and $K$ is a quadratic field satisfying $|\Delta_K| < 5 \cdot 10^5$. Conditional on ERH, $\mathcal{H}(K,1,\ell) = \varnothing$.
\end{proposition}

\appendix

\section{Tables of heavenly CM elliptic curves}\label{App:Tables}

Tables \ref{table:isog_data_1} and \ref{table:all_non_rational_j_curves}  list, up to isomorphism over the base field, every elliptic curve $E$ satisfying the following conditions:
\begin{itemize}
  \item $E$ is defined over a real quadratic field $K$ with $K = \QQ(j(E))$,
  \item $E$ has complex multiplication,
  \item $E$ is heavenly at some prime $\ell$.
\end{itemize}
Individual curves are listed in Table \ref{table:all_non_rational_j_curves}, organized by isogeny class. Data on each such class is presented in Table \ref{table:isog_data_1}.
Each isogeny class has a label of the form $H(\ell, \Delta).n$. Every curve in this class is heavenly at $\ell$ if $\ell$ is prime, and is heavenly at both $2$ and $3$ if $\ell = 6$.

In Table \ref{table:isog_data_1}, we present values for $\nu$, $m$, $\mathfrak{f}$, $r_{K}$, and $r_{\QQ}$. These are defined as follows: if the curves in the class have conductor $\mf{N}$, then $\nu$ satisfies $\ell^{\nu} = \mathbf{N}_{K/\QQ}\mf{N}$. (Note that curves in $H(6,24).1$ have everywhere good reduction, and so $\nu = 0$ for $\ell = 2$ or $\ell = 3$.) The field of complex multiplication for the isogeny class is $L = \QQ(\sqrt{-m})$. Each curve in the isogeny class has complex multiplication by some order $\mathcal{O}$ in $L$; the possible conductors for those orders are listed under $\mf{f}$. The isogeny class contains, up to $K$-isomorphism, $r_K$ curves with irrational $j$-invariant, and $r_{\QQ}$ curves with $\QQ$-rational $j$-invariant. 

Table \ref{table:all_non_rational_j_curves} lists the curves in each isogeny class which have $j(E) \not\in \QQ$; curves in the class with $j(E) \in \QQ$ are not listed. In Table \ref{table:all_non_rational_j_curves}, $a$ denotes a generator for the field of definition $K$. If $K$ has discriminant $\Delta$, then the minimal polynomial $m_a(T)$ for $a$ is
\[ m_a(T) = \begin{cases}
    T^2 - \left( \tfrac{\Delta}{4} \right) & \Delta \not\equiv 1 \bmod{4}, \\
    T^2 - T - \left( \tfrac{\Delta - 1}{4} \right) & \Delta \equiv 1 \bmod{4}. \end{cases} \]

\clearpage

\begin{table}[ht!]
  \centering
  \tiny
  \begin{tabular}{c|c}
  \begin{tabular}{rrrrrrrrr}
    \multicolumn{3}{l}{$\Delta_K = 5$} && \multicolumn{5}{r}{(\textsc{lmfdb} \href{https://www.lmfdb.org/NumberField/2.2.5.1}{\texttt{2.2.5.1}})} \\
    \midrule[0.7pt]
    $\ell$ && \multicolumn{1}{c}{Class} & \multicolumn{1}{c}{\textsc{lmfdb}} & $\nu$ & $m$ & $\mf{f}$ & $r_K$ & $r_{\QQ}$ \\
    \midrule
     \multirow{8}{*}{$2$} && $H(2,5).1$ & \href{https://www.lmfdb.org/EllipticCurve/2.2.5.1/4096.1/k/}{\texttt{4096.1-k}} & $12$ & $5$ & $1$ & $4$ & $0$ \\
  && $H(2,5).2$ & \href{https://www.lmfdb.org/EllipticCurve/2.2.5.1/4096.1/n/}{\texttt{4096.1-n}} & $12$ & $5$ & $1$ & $4$ & $0$ \\
  && $H(2,5).3$ & \href{https://www.lmfdb.org/EllipticCurve/2.2.5.1/4096.1/p/}{\texttt{4096.1-p}} & $12$ & $5$ & $1$ & $4$ & $0$ \\
  && $H(2,5).4$ & \href{https://www.lmfdb.org/EllipticCurve/2.2.5.1/4096.1/u/}{\texttt{4096.1-u}} & $12$ & $5$ & $1$ & $4$ & $0$ \\
    \cmidrule(lr){3-9}
  && $H(2,5).5$ && $16$ & $10$ & $1$ & $4$ & $0$ \\
  && $H(2,5).6$ && $16$ & $10$ & $1$ & $4$ & $0$ \\
  && $H(2,5).7$ && $16$ & $10$ & $1$ & $4$ & $0$ \\
  && $H(2,5).8$ && $16$ & $10$ & $1$ & $4$ & $0$ \\
    \cmidrule(lr){1-9}
    $3$ && $H(3,5).1$ & \href{https://www.lmfdb.org/EllipticCurve/2.2.5.1/81.1/a/}{\texttt{81.1-a}} & $4$ & $15$ & $1, 2$ & $8$ & $0$ \\
    \cmidrule(lr){1-9}
    $7$ && $H(7,5).1$ & \href{https://www.lmfdb.org/EllipticCurve/2.2.5.1/2401.1/b/}{\texttt{2401.1-b}} & $4$ & $35$ & $1$ & $4$ & $0$ \\
    \cmidrule(lr){1-9}
    $23$ && $H(23,5).1$ && $4$ & $115$ & $1$ & $4$ & $0$ \\
    \cmidrule(lr){1-9}
    $47$ && $H(47,5).1$ && $4$ & $235$ & $1$ & $4$ & $0$ \\
    \midrule[0.7pt]
    \noalign{\vspace{6pt}}
    \multicolumn{3}{l}{$\Delta_K = 8$} && \multicolumn{5}{r}{(\textsc{lmfdb} \href{https://www.lmfdb.org/NumberField/2.2.8.1}{\texttt{2.2.8.1}})} \\
    \midrule[0.7pt]
    $\ell$ && \multicolumn{1}{c}{Class} & \multicolumn{1}{c}{\textsc{lmfdb}} & $\nu$ & $m$ & $\mf{f}$ & $r_K$ & $r_{\QQ}$ \\
    \midrule
    \multirow{8}{*}{$2$} && $H(2,8).1$ & \href{https://www.lmfdb.org/EllipticCurve/2.2.8.1/64.1/a/}{\texttt{64.1-a}} & $6$ & $2$ & $2$ & $4$ & $2$ \\
    \cmidrule(lr){3-9}
    && $H(2,8).2$ & \href{https://www.lmfdb.org/EllipticCurve/2.2.8.1/256.1/a/}{\texttt{256.1-a}} & $8$ & $1$ & $4$ & $4$ & $4$ \\
    && $H(2,8).3$ & \href{https://www.lmfdb.org/EllipticCurve/2.2.8.1/256.1/c/}{\texttt{256.1-c}} & $8$ & $2$ & $2$ & $4$ & $2$ \\
    \cmidrule(lr){3-9}
    && $H(2,8).4$ & \href{https://www.lmfdb.org/EllipticCurve/2.2.8.1/1024.1/a/}{\texttt{1024.1-a}} & $10$ & $1$ & $4$ & $4$ & $4$ \\
    && $H(2,8).5$ & \href{https://www.lmfdb.org/EllipticCurve/2.2.8.1/1024.1/d/}{\texttt{1024.1-d}} & $10$ & $2$ & $2$ & $4$ & $2$ \\
    && $H(2,8).6$ & \href{https://www.lmfdb.org/EllipticCurve/2.2.8.1/1024.1/f/}{\texttt{1024.1-f}} & $10$ & $1$ & $4$ & $4$ & $4$ \\
    && $H(2,8).7$ & \href{https://www.lmfdb.org/EllipticCurve/2.2.8.1/1024.1/k/}{\texttt{1024.1-k}} & $10$ & $1$ & $4$ & $4$ & $4$ \\
    && $H(2,8).8$ & \href{https://www.lmfdb.org/EllipticCurve/2.2.8.1/1024.1/m/}{\texttt{1024.1-m}} & $10$ & $2$ & $2$ & $4$ & $2$ \\
    \cmidrule(lr){1-9}
    $3$ && $H(3,8).1$ & \href{https://www.lmfdb.org/EllipticCurve/2.2.8.1/81.1/b/}{\texttt{81.1-b}} & $4$ & $6$ & $1$ & $4$ & $0$ \\
    \cmidrule(lr){1-9}
    $11$ && $H(11, 8).1$ && $4$ & $22$ & $1$ & $4$ & $0$ \\
    \midrule[0.7pt]
    \noalign{\vspace{6pt}}
    \multicolumn{3}{l}{$\Delta_K = 12$} && \multicolumn{5}{r}{(\textsc{lmfdb} \href{https://www.lmfdb.org/NumberField/2.2.12.1}{\texttt{2.2.12.1}})} \\
    \midrule[0.7pt]
    $\ell$ && \multicolumn{1}{c}{Class} & \multicolumn{1}{c}{\textsc{lmfdb}} & $\nu$ & $m$ & $\mf{f}$ & $r_K$ & $r_{\QQ}$ \\
    \midrule
    \multirow{4}{*}{$2$} && $H(2, 12).1$ & \href{https://www.lmfdb.org/EllipticCurve/2.2.12.1/16.1/a/}{\texttt{16.1-a}} & $4$ & $3$ & $4$ & $4$ & $4$ \\
    \cmidrule(lr){3-9}
    && $H(2,12).2$ & \href{https://www.lmfdb.org/EllipticCurve/2.2.12.1/256.1/c/}{\texttt{256.1-c}} & $8$ & $3$ & $4$ & $4$ & $4$ \\
    \cmidrule(lr){3-9}
    && $H(2,12).3$ & \href{https://www.lmfdb.org/EllipticCurve/2.2.12.1/1024.1/j/}{\texttt{1024.1-j}} & $10$ & $3$ & $4$ & $4$ & $4$ \\
    && $H(2,12).4$ & \href{https://www.lmfdb.org/EllipticCurve/2.2.12.1/1024.1/k/}{\texttt{1024.1-k}} & $10$ & $3$ & $4$ & $4$ & $4$ \\
    \cmidrule(lr){1-9}
    $3$ && $H(3,12).1$ & \href{https://www.lmfdb.org/EllipticCurve/2.2.12.1/9.1/a/}{\texttt{9.1-a}} & $2$ & $1$ & $3$ & $4$ & $2$ \\
    \midrule[0.7pt]
    \noalign{\vspace{6pt}}
        \multicolumn{3}{l}{$\Delta_K = 13$} && \multicolumn{5}{r}{(\textsc{lmfdb} \href{https://www.lmfdb.org/NumberField/2.2.13.1}{\texttt{2.2.13.1}})} \\
    \midrule[0.7pt]
    $\ell$ && \multicolumn{1}{c}{Class} & \multicolumn{1}{c}{\textsc{lmfdb}} & $\nu$ & $m$ & $\mf{f}$ & $r_K$ & $r_{\QQ}$ \\
    \midrule
    \multirow{4}{*}{$2$} && $H(2, 13).1$ && $12$ & $13$ & $1$ & $4$ & $0$ \\
                         && $H(2, 13).2$ && $12$ & $13$ & $1$ & $4$ & $0$ \\
                         && $H(2, 13).3$ && $12$ & $13$ & $1$ & $4$ & $0$ \\
                         && $H(2, 13).4$ && $12$ & $13$ & $1$ & $4$ & $0$ \\
    \cmidrule(lr){1-9}
    $7$ && $H(7,13).1$ && $4$ & $91$ & $1$ & $4$ & $0$ \\
    \cmidrule(lr){1-9}
    $31$ && $H(31,13).1$ && $4$ & $403$ & $1$ & $4$ & $0$ \\
    \midrule[0.7pt]
    \noalign{\vspace{6pt}}
 \multicolumn{3}{l}{$\Delta_K = 17$} && \multicolumn{5}{r}{(\textsc{lmfdb} \href{https://www.lmfdb.org/NumberField/2.2.17.1}{\texttt{2.2.17.1}})} \\
    \midrule[0.7pt]
    $\ell$ && \multicolumn{1}{c}{Class} & \multicolumn{1}{c}{\textsc{lmfdb}} & $\nu$ & $m$ & $\mf{f}$ & $r_K$ & $r_{\QQ}$ \\
    \midrule
   $3$ && $H(3,17).1$ & \href{https://www.lmfdb.org/EllipticCurve/2.2.17.1/81.1/b/}{\texttt{81.1-b}} & $4$ & $51$ & $1$ & $4$ & $0$ \\
   \cmidrule(lr){1-9}
    $11$ && $H(11,17).1$ && $4$ & $187$ & $1$ & $4$ & $0$ \\
   \midrule[0.7pt]
   \noalign{\vspace{6pt}}
 \multicolumn{3}{l}{$\Delta_K = 21$} && \multicolumn{5}{r}{(\textsc{lmfdb} \href{https://www.lmfdb.org/NumberField/2.2.21.1}{\texttt{2.2.21.1}})} \\
    \midrule[0.7pt]
    $\ell$ && \multicolumn{1}{c}{Class} & \multicolumn{1}{c}{\textsc{lmfdb}} & $\nu$ & $m$ & $\mf{f}$ & $r_K$ & $r_{\QQ}$ \\
   \midrule
   $7$ && $H(7,21).1$ & \href{https://www.lmfdb.org/EllipticCurve/2.2.21.1/49.1/a/}{\texttt{49.1-a}} & $2$ & $3$ & $7$ & $4$ & $2$ \\
    \bottomrule
  \end{tabular}
&
 \begin{tabular}{rrrrrrrrr}
 \multicolumn{3}{l}{$\Delta_K = 24$} && \multicolumn{5}{r}{(\textsc{lmfdb} \href{https://www.lmfdb.org/NumberField/2.2.24.1}{\texttt{2.2.24.1}})} \\
    \midrule[0.7pt]
    $\ell$ && \multicolumn{1}{c}{Class} & \multicolumn{1}{c}{\textsc{lmfdb}} & $\nu$ & $m$ & $\mf{f}$ & $r_K$ & $r_{\QQ}$ \\
    \midrule
   \multirow{3}{*}{$2$} && $H(2,24).1$ & \href{https://www.lmfdb.org/EllipticCurve/2.2.24.1/16.1/a/}{\texttt{16.1-a}} & $4$ & $2$ & $3$ & $4$ & $2$ \\
   \cmidrule(lr){3-9}
   && $H(2,24).2$ && $10$ & $2$ & $3$ & $4$ & $2$ \\
   && $H(2,24).3$ && $10$ & $2$ & $3$ & $4$ & $2$ \\
   \cmidrule(lr){1-9}
   $2, 3$ && $H(6, 24).1$ & \href{https://www.lmfdb.org/EllipticCurve/2.2.24.1/1.1/a/}{\texttt{1.1-a}} & $0$ & $2$ & $3$ & $4$ & $2$ \\
   \midrule[0.7pt]
   \noalign{\vspace{6pt}}
 \multicolumn{3}{l}{$\Delta_K = 28$} && \multicolumn{5}{r}{(\textsc{lmfdb} \href{https://www.lmfdb.org/NumberField/2.2.28.1}{\texttt{2.2.28.1}})} \\
    \midrule[0.7pt]
    $\ell$ && \multicolumn{1}{c}{Class} & \multicolumn{1}{c}{\textsc{lmfdb}} & $\nu$ & $m$ & $\mf{f}$ & $r_K$ & $r_{\QQ}$ \\
    \midrule
   \multirow{4}{*}{$2$} && $H(2,28).1$ & \href{https://www.lmfdb.org/EllipticCurve/2.2.28.1/1.1/a/}{\texttt{16.1-a}} & $0$ & $7$ & $4$ & $4$ & $4$ \\
   \cmidrule(lr){3-9}
   && $H(2,28).2$ & \href{https://www.lmfdb.org/EllipticCurve/2.2.28.1/256.1/j/}{\texttt{256.1-j}}  & $8$ & $7$ & $4$ & $4$ & $4$ \\
   \cmidrule(lr){3-9}
   && $H(2,28).3$ && $10$ & $7$ & $4$ & $4$ & $4$ \\
   && $H(2,28).4$ && $10$ & $7$ & $4$ & $4$ & $4$ \\
   \cmidrule(lr){1-9}
   $7$ && $H(7,28).1$ & \href{https://www.lmfdb.org/EllipticCurve/2.2.28.1/49.1/a/}{\texttt{49.1-a}} & $2$ & $7$ & $4$ & $4$ & $4$ \\
   \midrule[0.7pt]
   \noalign{\vspace{6pt}}
 \multicolumn{3}{l}{$\Delta_K = 29$} && \multicolumn{5}{r}{(\textsc{lmfdb} \href{https://www.lmfdb.org/NumberField/2.2.29.1}{\texttt{2.2.29.1}})} \\
    \midrule[0.7pt]
    $\ell$ && \multicolumn{1}{c}{Class} & \multicolumn{1}{c}{\textsc{lmfdb}} & $\nu$ & $m$ & $\mf{f}$ & $r_K$ & $r_{\QQ}$ \\
    \midrule
   \multirow{4}{*}{$2$} && $H(2,29).1$ && $16$ & $58$ & $1$ & $4$ & $0$ \\
    && $H(2,29).2$ && $16$ & $58$ & $1$ & $4$ & $0$ \\
    && $H(2,29).3$ && $16$ & $58$ & $1$ & $4$ & $0$ \\
    && $H(2,29).4$ && $16$ & $58$ & $1$ & $4$ & $0$ \\
   \midrule[0.7pt]
   \noalign{\vspace{6pt}}
 \multicolumn{3}{l}{$\Delta_K = 33$} && \multicolumn{5}{r}{(\textsc{lmfdb} \href{https://www.lmfdb.org/NumberField/2.2.33.1}{\texttt{2.2.33.1}})} \\
    \midrule[0.7pt]
    $\ell$ && \multicolumn{1}{c}{Class} & \multicolumn{1}{c}{\textsc{lmfdb}} & $\nu$ & $m$ & $\mf{f}$ & $r_K$ & $r_{\QQ}$ \\
    \midrule
   $3$ && $H(3, 33).1$ & \href{https://www.lmfdb.org/EllipticCurve/2.2.33.1/1.1/a/}{\texttt{1.1-a}} & $0$ & $11$ & $3$ & $4$ & $2$ \\
   \cmidrule(lr){1-9}
   $11$ && $H(11, 33).1$ & \href{https://www.lmfdb.org/EllipticCurve/2.2.33.1/121.1/b/}{\texttt{121.1-b}} & $2$ & $11$ & $3$ & $4$ & $2$ \\
   \midrule[0.7pt]
   \noalign{\vspace{6pt}}
 \multicolumn{3}{l}{$\Delta_K = 37$} && \multicolumn{5}{r}{(\textsc{lmfdb} \href{https://www.lmfdb.org/NumberField/2.2.37.1}{\texttt{2.2.37.1}})} \\
    \midrule[0.7pt]
    $\ell$ && \multicolumn{1}{c}{Class} & \multicolumn{1}{c}{\textsc{lmfdb}} & $\nu$ & $m$ & $\mf{f}$ & $r_K$ & $r_{\QQ}$ \\
    \midrule
   \multirow{4}{*}{$2$} && $H(2,37).1$ && $12$ & $37$ & $1$ & $4$ & $0$ \\
    && $H(2,37).2$ && $12$ & $37$ & $1$ & $4$ & $0$ \\
    && $H(2,37).3$ && $12$ & $37$ & $1$ & $4$ & $0$ \\
    && $H(2,37).4$ && $12$ & $37$ & $1$ & $4$ & $0$ \\
   \midrule[0.7pt]
   \noalign{\vspace{6pt}}
 \multicolumn{3}{l}{$\Delta_K = 41$} && \multicolumn{5}{r}{(\textsc{lmfdb} \href{https://www.lmfdb.org/NumberField/2.2.41.1}{\texttt{2.2.41.1}})} \\
    \midrule[0.7pt]
    $\ell$ && \multicolumn{1}{c}{Class} & \multicolumn{1}{c}{\textsc{lmfdb}} & $\nu$ & $m$ & $\mf{f}$ & $r_K$ & $r_{\QQ}$ \\
    \midrule
   $3$ && $H(3, 41).1$ & \href{https://www.lmfdb.org/EllipticCurve/2.2.41.1/81.1/c/}{\texttt{81.1-c}} & $4$ & $123$ & $1$ & $4$ & $0$ \\
   \midrule[0.7pt]
   \noalign{\vspace{6pt}}
 \multicolumn{3}{l}{$\Delta_K = 61$} && \multicolumn{5}{r}{(\textsc{lmfdb} \href{https://www.lmfdb.org/NumberField/2.2.61.1}{\texttt{2.2.61.1}})} \\
    \midrule[0.7pt]
    $\ell$ && \multicolumn{1}{c}{Class} & \multicolumn{1}{c}{\textsc{lmfdb}} & $\nu$ & $m$ & $\mf{f}$ & $r_K$ & $r_{\QQ}$ \\
    \midrule
   $7$ && $H(7, 61).1$ && $4$ & $427$ & $1$ & $4$ & $0$ \\
   \midrule[0.7pt]
   \noalign{\vspace{6pt}}
 \multicolumn{3}{l}{$\Delta_K = 89$} && \multicolumn{5}{r}{(\textsc{lmfdb} \href{https://www.lmfdb.org/NumberField/2.2.89.1}{\texttt{2.2.89.1}})} \\
    \midrule[0.7pt]
    $\ell$ && \multicolumn{1}{c}{Class} & \multicolumn{1}{c}{\textsc{lmfdb}} & $\nu$ & $m$ & $\mf{f}$ & $r_K$ & $r_{\QQ}$ \\
    \midrule
   $3$ && $H(3, 89).1$ & \href{https://www.lmfdb.org/EllipticCurve/2.2.89.1/81.1/a/}{\texttt{81.1-a}} & $4$ & $267$ & $1$ & $4$ & $0$ \\
   \bottomrule
   \noalign{\vspace{70.25pt}}
 \end{tabular}
\end{tabular}
\caption{Isogeny classes of heavenly elliptic curves over quadratic fields with complex multiplication}\label{table:isog_data_1}
\end{table}

\clearpage

\gdef\z{\phantom{0}}

{\small 
\begin{longtable}[c]{r rrrrr}
\caption{Heavenly curves over quadratic fields with complex multiplication, $j \not\in \QQ$}\label{table:all_non_rational_j_curves} \\
  \toprule
Class & $a_1$ & $a_2$ & $a_3$ & $a_4$ & $a_6$ \endfirsthead
  \toprule
Class & $a_1$ & $a_2$ & $a_3$ & $a_4$ & $a_6$ \endhead 
\midrule[1pt]
 $H(2,5).1$ 
 & $0$ & $-a + 1$ & $0$ & $-2$ & $2 a - 4$ \\ 
 & $0$ & $a - 1$ & $0$ & $-2$ & $-2 a + 4$ \\ 
 & $0$ & $-a$ & $0$ & $-a - 9$ & $-6 a - 15$ \\ 
 & $0$ & $a$ & $0$ & $-a - 9$ & $6 a + 15$ \\ 
 \cmidrule(lr){1-6}
 $H(2,5).2$
 & $0$ & $a$ & $0$ & $-2$ & $-2 a - 2$ \\ 
 & $0$ & $-a$ & $0$ & $-2$ & $2 a + 2$ \\ 
 & $0$ & $a - 1$ & $0$ & $a - 10$ & $6 a - 21$ \\ 
 & $0$ & $-a + 1$ & $0$ & $a - 10$ & $-6 a + 21$ \\
 \cmidrule(lr){1-6}
 $H(2,5).3$
 & $0$ & $1$ & $0$ & $-2 a - 2$ & $2 a$ \\ 
 & $0$ & $-1$ & $0$ & $-2 a - 2$ & $-2 a$ \\ 
 & $0$ & $1$ & $0$ & $8 a - 17$ & $24 a - 33$ \\ 
 & $0$ & $-1$ & $0$ & $8 a - 17$ & $-24 a + 33$ \\
 \cmidrule(lr){1-6}
 $H(2,5).4$
 & $0$ & $-1$ & $0$ & $2 a - 4$ & $2 a - 2$ \\ 
 & $0$ & $1$ & $0$ & $2 a - 4$ & $-2 a + 2$ \\ 
 & $0$ & $-1$ & $0$ & $-8 a - 9$ & $24 a + 9$ \\ 
 & $0$ & $1$ & $0$ & $-8 a - 9$ & $-24 a - 9$ \\ 
 \cmidrule(lr){1-6}
 $H(2,5).5$
 & $0$ & $0$ & $0$ & $6 a - 28$ & $-16 a + 56$ \\ 
 & $0$ & $0$ & $0$ & $-6 a - 22$ & $-16 a - 40$ \\ 
 & $0$ & $0$ & $0$ & $-24 a - 88$ & $128 a + 320$ \\ 
 & $0$ & $0$ & $0$ & $24 a - 112$ & $128 a - 448$ \\ 
 \cmidrule(lr){1-6}
 $H(2,5).6$
 & $0$ & $0$ & $0$ & $6 a - 28$ & $16 a - 56$ \\ 
 & $0$ & $0$ & $0$ & $-6 a - 22$ & $16 a + 40$ \\ 
 & $0$ & $0$ & $0$ & $-24 a - 88$ & $-128 a - 320$ \\ 
 & $0$ & $0$ & $0$ & $24 a - 112$ & $-128 a + 448$ \\ 
 \cmidrule(lr){1-6}
 $H(2,5).7$
 & $0$ & $0$ & $0$ & $-16 a - 22$ & $-64 a - 24$ \\ 
 & $0$ & $0$ & $0$ & $16 a - 38$ & $64 a - 88$ \\ 
 & $0$ & $0$ & $0$ & $-64 a - 88$ & $512 a + 192$ \\ 
 & $0$ & $0$ & $0$ & $64 a - 152$ & $-512 a + 704$ \\ 
 \cmidrule(lr){1-6}
 $H(2,5).8$
 & $0$ & $0$ & $0$ & $-16 a - 22$ & $64 a + 24$ \\ 
 & $0$ & $0$ & $0$ &  $16 a - 38$ & $-64 a + 88$ \\ 
 & $0$ & $0$ & $0$ & $-64 a - 88$ & $-512 a - 192$ \\ 
 & $0$ & $0$ & $0$ &  $64 a - 152$ & $512 a - 704$ \\ 
\midrule[1pt]
 $H(3,5).1$
 &     $1$ &     $-1$ &     $a$ &   $-2 a$       & $a$ \\ 
 &     $1$ &     $-1$ & $a + 1$ &      $a -   2$ &   $-2 a +    1$ \\ 
 &     $1$ &     $-1$ &     $a$ &   $13 a -  15$ &   $20 a -   26$ \\ 
 &     $1$ &     $-1$ & $a + 1$ &  $-14 a -   2$ &  $-21 a -    6$ \\ 
 & $a + 1$ &      $1$ &     $1$ &   $13 a -  26$ &   $32 a -   51$ \\ 
 &     $a$ & $-a - 1$ & $a + 1$ &  $-13 a -  14$ &  $-20 a -    6$ \\ 
 & $a + 1$ &      $1$ &     $0$ &  $114 a - 237$ & $-754 a + 1014$ \\ 
 &     $a$ & $-a - 1$ &     $a$ & $-113 a - 125$ &  $867 a +  384$ \\ 
 \midrule[1pt]
 $H(7,5).1$
 & $0$ & $-1$ & $1$ & $14 a + \z 5$ & $-21 a - 15$ \\ 
 & $0$ & $-1$ & $1$ & $-14 a + 19$ & $21 a - 36$ \\ 
 & $0$ & $1$ & $1$ & $686 a + 229$ & $5831 a + 4589$ \\ 
 & $0$ & $1$ & $1$ & $-686 a + 915$ & $-5831 a + 10420$ \\ 
 \midrule[1pt]
 \pagebreak
 \midrule[1pt]
 $H(23,5).1$
 & $0$ & $0$ & $1$ & $46 a - 368$ & $2645 a - 6216$ \\ 
 & $0$ & $0$ & $1$ & $-46 a - 322$ & $-2645 a - 3571$ \\ 
 & $0$ & $0$ & $1$ & $24334 a - 194672$ & $-32181715 a + 75627030$ \\ 
 & $0$ & $0$ & $1$ & $-24334 a - 170338$ & $32181715 a + 43445315$ \\ 
 \midrule[1pt]
 $H(47,5).1$ 
 & $0$ & $0$ & $1$ & $4136 a - 17578$ & $324723 a - 962572$ \\ 
 & $0$ & $0$ & $1$ & $-4136 a - 13442$ & $-324723 a - 637849$ \\ 
 & $0$ & $0$ & $1$ & $9136424 a - 38829802$ & $-33713716029 a + 99937086800$ \\ 
 & $0$ & $0$ & $1$ & $-9136424 a - 29693378$ & $33713716029 a + 66223370771$ \\ 
 \midrule[1pt]
 $H(2,8).1$
 & $a$ & $a$ & $0$ & $2 a + 2$ & $3 a + 5$ \\ 
 & $a$ & $-a$ & $0$ & $-2 a + 2$ & $-3 a + 5$ \\ 
 & $a$ & $-a - 1$ & $0$ & $2 a + 2$ & $-3 a - 5$ \\ 
 & $a$ & $a - 1$ & $0$ & $-2 a + 2$ & $3 a - 5$ \\  
 \cmidrule(lr){1-6}
 $H(2,8).2$
 & $0$ & $0$ & $0$ & $2 a - 33$ & $154 a - 154$ \\ 
 & $0$ & $0$ & $0$ & $2 a - 33$ & $-154 a + 154$ \\ 
 & $0$ & $0$ & $0$ & $-2 a - 33$ & $154 a + 154$ \\ 
 & $0$ & $0$ & $0$ & $-2 a - 33$ & $-154 a - 154$ \\ 
 \cmidrule(lr){1-6}
 $H(2,8).3$
 & $0$ & $a$ & $0$ & $10 a - 11$ & $-23 a + 30$ \\ 
 & $0$ & $-a$ & $0$ & $10 a - 11$ & $23 a - 30$ \\ 
 & $0$ & $a$ & $0$ & $-10 a - 11$ & $-23 a - 30$ \\ 
 & $0$ & $-a$ & $0$ & $-10 a - 11$ & $23 a + 30$ \\  
 \cmidrule(lr){1-6}
  $H(2,8).4$
 & $a$ & $1$ & $a$ & $15 a - 23$ & $-31 a + 46$ \\ 
 & $a$ & $1$ & $0$ & $15 a - 22$ & $46 a - 69$ \\ 
 & $a$ & $1$ & $a$ & $-15 a - 23$ & $31 a + 46$ \\ 
 & $a$ & $1$ & $0$ & $-15 a - 22$ & $-46 a - 69$ \\ 
 \cmidrule(lr){1-6}
 $H(2,8).5$
 & $0$ & $1$ & $0$ & $20 a - 23$ & $-40 a + 69$ \\ 
 & $0$ & $-1$ & $0$ & $20 a - 23$ & $40 a - 69$ \\ 
 & $0$ & $-1$ & $0$ & $-20 a - 23$ & $-40 a - 69$ \\ 
 & $0$ & $1$ & $0$ & $-20 a - 23$ & $40 a + 69$ \\ 
 \cmidrule(lr){1-6}
 $H(2,8).6$
 & $0$ & $0$ & $0$ & $4 a - 66$ & $308 a - 616$ \\ 
 & $0$ & $0$ & $0$ & $4 a - 66$ & $-308 a + 616$ \\ 
 & $0$ & $0$ & $0$ & $-4 a - 66$ & $308 a + 616$ \\ 
 & $0$ & $0$ & $0$ & $-4 a - 66$ & $-308 a - 616$ \\ 
 \cmidrule(lr){1-6}
 $H(2,8).7$
 & $0$ & $0$ & $0$ & $120 a - 182$ & $924 a - 1232$ \\ 
 & $0$ & $0$ & $0$ & $120 a - 182$ & $-924 a + 1232$ \\ 
 & $0$ & $0$ & $0$ & $-120 a - 182$ & $-924 a - 1232$ \\ 
 & $0$ & $0$ & $0$ & $-120 a - 182$ & $924 a + 1232$ \\ 
 \cmidrule(lr){1-6}
 $H(2,8).8$
 & $0$ & $a + 1$ & $0$ & $14 a + 11$ & $65 a + 83$ \\ 
 & $0$ & $-a - 1$ & $0$ & $14 a + 11$ & $-65 a - 83$ \\ 
 & $0$ & $-a + 1$ & $0$ & $-14 a + 11$ & $-65 a + 83$ \\ 
 & $0$ & $a - 1$ & $0$ & $-14 a + 11$ & $65 a - 83$ \\ 
\midrule[1pt]
 $H(3,8).1$
 & $a$ & $1$ & $1$ & $a - 3$ & $-a + 1$ \\ 
 & $a$ & $1$ & $1$ & $-2 a - 3$ & $a + 1$ \\ 
 & $a$ & $1$ & $1$ & $13 a - 33$ & $54 a - 98$ \\ 
 & $a$ & $1$ & $1$ & $-14 a - 33$ & $-54 a - 98$ \\ 
\midrule[1pt]
 $H(11,8).1$
 & $a$ & $1$ & $1$ & $192 a - 453$ & $-2233 a + 4008$ \\ 
 & $a$ & $1$ & $1$ & $-193 a - 453$ & $2233 a + 4008$ \\ 
 & $a$ & $1$ & $1$ & $23292 a - 54903$ & $3111878 a - 5664237$ \\ 
 & $a$ & $1$ & $1$ & $-23293 a - 54903$ & $-3111878 a - 5664237$ \\ 
\midrule[1pt]
\pagebreak
\midrule[1pt]
 $H(2,12).1$
 & $a + 1$ & $-a - 1$ & $a + 1$ &  $4 a - 13$ &  $11 a - 21$ \\ 
 & $a + 1$ &     $-1$ & $a + 1$ &  $4 a - 13$ & $-12 a + 19$ \\ 
 & $a + 1$ & $-a - 1$ & $a + 1$ & $-6 a - 13$ &  $11 a + 19$ \\ 
 & $a + 1$ &     $-1$ & $a + 1$ & $-6 a - 13$ & $-12 a - 21$ \\ 
 \cmidrule(lr){1-6}
 $H(2,12).2$
 & $0$ & $-a$ & $0$ &  $20 a - 44$ &  $92 a - 160$ \\ 
 & $0$ &  $a$ & $0$ &  $20 a - 44$ & $-92 a + 160$ \\ 
 & $0$ & $-a$ & $0$ & $-20 a - 44$ &  $92 a + 160$ \\ 
 & $0$ &  $a$ & $0$ & $-20 a - 44$ & $-92 a - 160$ \\ 
 \cmidrule(lr){1-6}
 $H(2,12).3$
 & $0$ &  $a$ & $0$ &  $-10 a -  59$ &    $49 a +    2$ \\ 
 & $0$ & $-a$ & $0$ &  $-10 a -  59$ &   $-49 a -    2$ \\ 
 & $0$ & $-a$ & $0$ & $-170 a - 299$ &  $1711 a + 2958$ \\ 
 & $0$ &  $a$ & $0$ & $-170 a - 299$ & $-1711 a - 2958$ \\ 
 \cmidrule(lr){1-6}
 $H(2,12).4$
 & $0$ & $a$ & $0$ & $10 a - 59$ & $49 a - 2$ \\ 
 & $0$ & $-a$ & $0$ & $10 a - 59$ & $-49 a + 2$ \\ 
 & $0$ & $-a$ & $0$ & $170 a - 299$ & $1711 a - 2958$ \\ 
 & $0$ & $a$ & $0$ & $170 a - 299$ & $-1711 a + 2958$ \\ 
\midrule[1pt]
 $H(3,12).1$
 & $a + 1$ & $a - 1$ & $a$ & $25 a - 45$ & $72 a - 127$ \\ 
 & $a + 1$ & $a - 1$ & $1$ & $-26 a - 44$ & $72 a + 126$ \\ 
 & $a + 1$ & $a - 1$ & $a$ & $25 a - 45$ & $-117 a + 202$ \\ 
 & $a + 1$ & $a - 1$ & $1$ & $-26 a - 44$ & $-117 a - 203$ \\ 
 \midrule[1pt]
 $H(2,13).1$
 & $0$ & $0$ & $0$ & $-35 a - 50$ & $160 a + 212$ \\ 
 & $0$ & $0$ & $0$ & $-35 a - 50$ & $-160 a - 212$ \\ 
 & $0$ & $0$ & $0$ & $140 a - 340$ & $1280 a - 2976$ \\ 
 & $0$ & $0$ & $0$ & $140 a - 340$ & $-1280 a + 2976$ \\ 
 \cmidrule(lr){1-6}
 $H(2,13).2$
 & $0$ & $0$ & $0$ & $10 a - 35$ & $40 a - 76$ \\ 
 & $0$ & $0$ & $0$ & $10 a - 35$ & $-40 a + 76$ \\ 
 & $0$ & $0$ & $0$ & $-40 a - 100$ & $320 a + 288$ \\ 
 & $0$ & $0$ & $0$ & $-40 a - 100$ & $-320 a - 288$ \\ 
 \cmidrule(lr){1-6}
 $H(2,13).3$
 & $0$ & $0$ & $0$ & $35 a - 85$ & $160 a - 372$ \\ 
 & $0$ & $0$ & $0$ & $35 a - 85$ & $-160 a + 372$ \\ 
 & $0$ & $0$ & $0$ & $-140 a - 200$ & $1280 a + 1696$ \\ 
 & $0$ & $0$ & $0$ & $-140 a - 200$ & $-1280 a - 1696$ \\ 
\cmidrule(lr){1-6}
 $H(2,13).4$
 & $0$ & $0$ & $0$ & $-10 a - 25$ & $40 a + 36$ \\ 
 & $0$ & $0$ & $0$ & $-10 a - 25$ & $-40 a - 36$ \\ 
 & $0$ & $0$ & $0$ & $40 a - 140$ & $320 a - 608$ \\ 
 & $0$ & $0$ & $0$ & $40 a - 140$ & $-320 a + 608$ \\ 
 \midrule[1pt]
 $H(7,13).1$
 & $0$ & $0$ & $1$ & $84 a - 182$ & $539 a - 1213$ \\ 
 & $0$ & $0$ & $1$ & $-84 a - \z 98$ & $-539 a - \z 674$ \\ 
 & $0$ & $0$ & $1$ & $4116 a - 8918$ & $-184877 a + 415973$ \\ 
 & $0$ & $0$ & $1$ & $-4116 a - 4802$ & $184877 a + 231096$ \\ 
\midrule[1pt] 
 $H(31,13).1$
 & $0$ & $0$ & $1$ & $186930 a - 427490$ & $58571989 a - 135261471$ \\ 
 & $0$ & $0$ & $1$ & $-186930 a - 240560$ & $-58571989 a - \z 76689482$ \\ 
 & $0$ & $0$ & $1$ & $179639730 a - 410817890$ & $-1744918124299 a + 4029574475113$ \\ 
 & $0$ & $0$ & $1$ & $-179639730 a - 231178160$ & $1744918124299 a + 2284656350814$ \\ 
 \midrule[1pt]
 $H(3,17).1$
 & $0$ & $0$ & $1$ & $6 a - 18$ & $-14 a + 33$ \\ 
 & $0$ & $0$ & $1$ & $-6 a - 12$ & $14 a + 19$ \\ 
 & $0$ & $0$ & $1$ & $54 a - 162$ & $378 a - 898$ \\ 
 & $0$ & $0$ & $1$ & $-54 a - 108$ & $-378 a - 520$ \\ 
\midrule[1pt]
 \pagebreak
 \midrule[1pt]
 $H(11,17).1$
 & $0$ & $0$ & $1$ & $1430 a - 3520$ & $-40898 a + 104090$ \\ 
 & $0$ & $0$ & $1$ & $-1430 a - 2090$ & $40898 a + \z 63192$ \\ 
 & $0$ & $0$ & $1$ & $173030 a - 425920$ & $54435238 a - 138544123$ \\ 
 & $0$ & $0$ & $1$ & $-173030 a - 252890$ & $-54435238 a - \z 84108885$ \\ 
\midrule[1pt]
 $H(7,21).1$
 & $0$ & $-a - 1$ & $1$ & $281 a - 1608$ & $-5819 a + 26465$ \\ 
 & $0$ & $a + 1$ & $1$ & $-279 a - 1328$ & $5539 a + 19318$ \\ 
 & $0$ & $-a - 1$ & $1$ & $4131 a - 11618$ & $221331 a - 618025$ \\ 
 & $0$ & $a + 1$ & $1$ & $-4129 a - \z 7488$ & $-225461 a - 404182$ \\ 
\midrule[1pt]
 $H(2,24).1$
 & $0$ & $-a - 1$ & $0$ & $64 a - 166$ & $-418 a + 1056$ \\ 
 & $0$ & $a - 1$ & $0$ & $-64 a - 166$ & $418 a + 1056$ \\ 
 & $0$ & $-a - 1$ & $0$ & $264 a - 646$ & $3782 a - 9264$ \\ 
 & $0$ & $a - 1$ & $0$ & $-264 a - 646$ & $-3782 a - 9264$ \\ 
\cmidrule(lr){1-6}
 $H(2,24).2$
 & $0$ & $1$ & $0$ & $40 a - 163$ & $-360 a + 657$ \\ 
 & $0$ & $1$ & $0$ & $-40 a - 163$ & $360 a + 657$ \\ 
 & $0$ & $-a + 1$ & $0$ & $1306 a - 3201$ & $-38705 a + 94807$ \\ 
 & $0$ & $a + 1$ & $0$ & $-1306 a - 3201$ & $38705 a + 94807$ \\ 
\cmidrule(lr){1-6}
 $H(2,24).3$
 & $0$ & $-1$ & $0$ & $40 a - 163$ & $360 a - 657$ \\ 
 & $0$ & $-1$ & $0$ & $-40 a - 163$ & $-360 a - 657$ \\ 
 & $0$ & $a - 1$ & $0$ & $1306 a - 3201$ & $38705 a - 94807$ \\ 
 & $0$ & $-a - 1$ & $0$ & $-1306 a - 3201$ & $-38705 a - 94807$ \\
\midrule[1pt]
 $H(6,24).1$
 & $a$ & $a + 1$ & $a + 1$ & $17 a - 41$ & $-57 a + 138$ \\ 
 & $a$ & $-a + 1$ & $a + 1$ & $-18 a - 41$ & $56 a + 138$ \\ 
 & $a$ & $a + 1$ & $a + 1$ & $67 a - 161$ & $458 a - 1122$ \\ 
 & $a$ & $-a + 1$ & $a + 1$ & $-68 a - 161$ & $-459 a - 1122$ \\ 
\midrule[1pt]
 $H(2,28).1$
 & $1$ & $-1$ & $1$ & $270 a - 715$ & $-3223 a + 8527$ \\ 
 & $a$ & $-1$ & $a$ & $270 a - 718$ & $3223 a - 8529$ \\ 
 & $1$ & $-1$ & $1$ & $-270 a - 715$ & $3223 a + 8527$ \\ 
 & $a$ & $-1$ & $a$ & $-270 a - 718$ & $-3223 a - 8529$ \\ 
\cmidrule(lr){1-6}
 $H(2,28).2$
 & $0$ & $0$ & $0$ & $240 a - 725$ & $3698 a - 9520$ \\ 
 & $0$ & $0$ & $0$ & $240 a - 725$ & $-3698 a + 9520$ \\ 
 & $0$ & $0$ & $0$ & $-240 a - 725$ & $-3698 a - 9520$ \\ 
 & $0$ & $0$ & $0$ & $-240 a - 725$ & $3698 a + 9520$ \\ 
\cmidrule(lr){1-6}
 $H(2,28).3$
 & $0$ & $0$ & $0$ & $510 a - 1520$ & $9140 a - 23324$ \\ 
 & $0$ & $0$ & $0$ & $510 a - 1520$ & $-9140 a + 23324$ \\ 
 & $0$ & $0$ & $0$ & $8190 a - 21680$ & $656500 a - 1736924$ \\ 
 & $0$ & $0$ & $0$ & $8190 a - 21680$ & $-656500 a + 1736924$ \\ 
\cmidrule(lr){1-6}
 $H(2,28).4$
 & $0$ & $0$ & $0$ & $-510 a - 1520$ & $9140 a + 23324$ \\ 
 & $0$ & $0$ & $0$ & $-510 a - 1520$ & $-9140 a - 23324$ \\ 
 & $0$ & $0$ & $0$ & $-8190 a - 21680$ & $656500 a + 1736924$ \\ 
 & $0$ & $0$ & $0$ & $-8190 a - 21680$ & $-656500 a - 1736924$ \\ 
 \midrule[1pt]
 $H(7,28).1$
 & $1$ & $-1$ & $0$ & $105 a - 317$ & $1015 a - 2752$ \\ 
 & $a$ & $-1$ & $0$ & $105 a - 317$ & $-1015 a + 2752$ \\ 
 & $a$ & $-1$ & $0$ & $-105 a - 317$ & $1015 a + 2752$ \\ 
 & $1$ & $-1$ & $0$ & $-105 a - 317$ & $-1015 a - 2752$ \\ 
 \midrule[1pt]
 $H(2,29).1$
 & $0$ & $0$ & $0$ & $65070 a - 209680$ & $-14942928 a + 47739384$ \\ 
 & $0$ & $0$ & $0$ & $-65070 a - 144610$ & $-14942928 a - 32796456$ \\ 
 & $0$ & $0$ & $0$ & $260280 a - 838720$ & $119543424 a - 381915072$ \\ 
 & $0$ & $0$ & $0$ & $-260280 a - 578440$ & $119543424 a + 262371648$ \\ 
 \cmidrule(lr){1-6}
 \pagebreak
\midrule[1pt]
 $H(2,29).2$
 & $0$ & $0$ & $0$ & $65070 a - 209680$ & $14942928 a - 47739384$ \\ 
 & $0$ & $0$ & $0$ & $-65070 a - 144610$ & $14942928 a + 32796456$ \\ 
 & $0$ & $0$ & $0$ & $260280 a - 838720$ & $-119543424 a + 381915072$ \\ 
 & $0$ & $0$ & $0$ & $-260280 a - 578440$ & $-119543424 a - 262371648$ \\ 
\cmidrule(lr){1-6}
 $H(2,29).3$
 & $0$ & $0$ & $0$ & $7280 a - 36310$ & $960960 a - 2492952$ \\ 
 & $0$ & $0$ & $0$ & $-7280 a - 29030$ & $-960960 a - 1531992$ \\ 
 & $0$ & $0$ & $0$ & $29120 a - 145240$ & $-7687680 a + 19943616$ \\ 
 & $0$ & $0$ & $0$ & $-29120 a - 116120$ & $7687680 a + 12255936$ \\ 
\cmidrule(lr){1-6}
 $H(2,29).4$
 & $0$ & $0$ & $0$ &   $7280 a -  36310$ &  $-960960 a +  2492952$ \\ 
 & $0$ & $0$ & $0$ &  $-7280 a -  29030$ &   $960960 a +  1531992$ \\ 
 & $0$ & $0$ & $0$ &  $29120 a - 145240$ &  $7687680 a - 19943616$ \\ 
 & $0$ & $0$ & $0$ & $-29120 a - 116120$ & $-7687680 a - 12255936$ \\ 
\midrule[1pt]
 $H(3,33).1$
 & $0$ & $-a$ & $1$ & $25 a + 60$ & $-72 a - 171$ \\ 
 & $0$ & $a - 1$ & $1$ & $-25 a + 85$ & $72 a - 243$ \\ 
 & $0$ & $a - 1$ & $1$ & $435 a - 1465$ & $7890 a - 26607$ \\ 
 & $0$ & $-a$ & $1$ & $-435 a - 1030$ & $-7890 a - 18717$ \\
 \midrule[1pt]
 $H(11,33).1$
 & $0$ & $-1$ & $1$ & $110 a - 337$ & $-1012 a + 3431$ \\ 
 & $0$ & $-1$ & $1$ & $-110 a - 227$ & $1012 a + 2419$ \\ 
 & $0$ & $-a + 1$ & $1$ & $12371 a + 29351$ & $-513055 a - 1217110$ \\ 
 & $0$ & $a$ & $1$ & $-12371 a + 41722$ & $513055 a - 1730165$ \\ 
\midrule[1pt]
 $H(2,37).1$
 & $0$ & $0$ & $0$ & $-14110 a - 35875$ & $1591520 a + 4044684$ \\ 
 & $0$ & $0$ & $0$ & $-14110 a - 35875$ & $-1591520 a - 4044684$ \\ 
 & $0$ & $0$ & $0$ & $56440 a - 199940$ & $12732160 a - 45089632$ \\ 
 & $0$ & $0$ & $0$ & $56440 a - 199940$ & $-12732160 a + 45089632$ \\ 
 \cmidrule(lr){1-6}
 $H(2,37).2$
 & $0$ & $0$ & $0$ & $290 a - 1615$ & $8120 a - 23268$ \\ 
 & $0$ & $0$ & $0$ & $290 a - 1615$ & $-8120 a + 23268$ \\ 
 & $0$ & $0$ & $0$ & $-1160 a - 5300$ & $64960 a + 121184$ \\ 
 & $0$ & $0$ & $0$ & $-1160 a - 5300$ & $-64960 a - 121184$ \\ 
 \cmidrule(lr){1-6}
 $H(2,37).3$
 & $0$ & $0$ & $0$ & $14110 a - 49985$ & $-1591520 a + 5636204$ \\ 
 & $0$ & $0$ & $0$ & $14110 a - 49985$ & $1591520 a - 5636204$ \\ 
 & $0$ & $0$ & $0$ & $-56440 a - 143500$ & $12732160 a + 32357472$ \\ 
 & $0$ & $0$ & $0$ & $-56440 a - 143500$ & $-12732160 a - 32357472$ \\ 
 \cmidrule(lr){1-6}
 $H(2,37).4$
 & $0$ & $0$ & $0$ & $-290 a - 1325$ & $8120 a + 15148$ \\ 
 & $0$ & $0$ & $0$ & $-290 a - 1325$ & $-8120 a - 15148$ \\ 
 & $0$ & $0$ & $0$ & $1160 a - 6460$ & $64960 a - 186144$ \\ 
 & $0$ & $0$ & $0$ & $1160 a - 6460$ & $-64960 a + 186144$ \\ 
 \midrule[1pt]
 $H(3,41).1$
 & $0$ & $0$ & $1$ & $-60 a - 210$ & $560 a + 1384$ \\ 
 & $0$ & $0$ & $1$ & $60 a - 270$ & $-560 a + 1944$ \\ 
 & $0$ & $0$ & $1$ & $540 a - 2430$ & $15120 a - 52495$ \\ 
 & $0$ & $0$ & $1$ & $-540 a - 1890$ & $-15120 a - 37375$ \\ 
 \midrule[1pt]
 $H(7,61).1$
 & $0$ & $0$ & $1$ & $30030 a - 137060$ & $5787145 a - 25355528$ \\ 
 & $0$ & $0$ & $1$ & $-30030 a - 107030$ & $-5787145 a - 19568383$ \\ 
 & $0$ & $0$ & $1$ & $1471470 a - 6715940$ & $-1984990735 a + 8696946018$ \\ 
 & $0$ & $0$ & $1$ & $-1471470 a - 5244470$ & $1984990735 a + 6711955283$ \\ 
 \midrule[1pt]
 $H(3,89).1$
 & $0$ & $0$ & $1$  &   $1590 a - 10170$ & $-92750 a +  452625$ \\ 
 & $0$ & $0$ & $1$  &  $-1590 a - 8580$ & $92750 a +  359875$ \\ 
 & $0$ & $0$ & $1$  &  $14310 a - 91530$ & $2504250 a - 12220882$ \\ 
 & $0$ & $0$ & $1$  & $-14310 a - 77220$ & $-2504250 a - 9716632$ \\ 
\bottomrule
\end{longtable}
}




\subsection*{Data and Code Availability}\mbox{}

The proofs of Proposition \ref{prop:B21_is_19} and Theorem \ref{thm:heavenly_cm_calculation} are supported by calculations performed in {\tt Sage} \cite{SageMath} and {\tt MAGMA} \cite{MAGMA}. Code for these calculations are available at the repository \cite{McLeman-Rasmussen:2026}. 




\printbibliography

@article {Anderson-Ihara:1988,
    AUTHOR = {Anderson, Greg and Ihara, Yasutaka},
     TITLE = {Pro-{$l$} branched coverings of {${\bf P}^1$} and higher
              circular {$l$}-units},
   JOURNAL = {Ann. of Math. (2)},
  FJOURNAL = {Annals of Mathematics. Second Series},
    VOLUME = {128},
      YEAR = {1988},
    NUMBER = {2},
     PAGES = {271--293},
      ISSN = {0003-486X,1939-8980},
   MRCLASS = {14H25 (11G20 14H30)},
  MRNUMBER = {960948},
MRREVIEWER = {David\ Goss},
       DOI = {10.2307/1971443},
       URL = {https://doi.org/10.2307/1971443},
}

@article {Arai-Momose:2014,
    AUTHOR = {Arai, Keisuke and Momose, Fumiyuki},
     TITLE = {Algebraic points on {S}himura curves of {$\Gamma_0(p)$}-type},
   JOURNAL = {J. Reine Angew. Math.},
  FJOURNAL = {Journal f\"ur die Reine und Angewandte Mathematik. [Crelle's
              Journal]},
    VOLUME = {690},
      YEAR = {2014},
     PAGES = {179--202},
      ISSN = {0075-4102,1435-5345},
   MRCLASS = {11G18},
  MRNUMBER = {3200341},
MRREVIEWER = {Riccardo\ Brasca},
       DOI = {10.1515/crelle-2012-0068},
       URL = {https://doi.org/10.1515/crelle-2012-0068},
}

@article{Bach-Sorenson:1996,
    AUTHOR = {Bach, Eric and Sorenson, Jonathan},
     TITLE = {Explicit bounds for primes in residue classes},
   JOURNAL = {Math.~Comp.},
  FJOURNAL = {Mathematics of Computation},
    VOLUME = {65},
      YEAR = {1996},
    NUMBER = {216},
     PAGES = {1717--1735},
}

@article {Bourdon:2015,
    AUTHOR = {Bourdon, Abbey},
     TITLE = {A uniform version of a finiteness conjecture for {CM} elliptic
              curves},
   JOURNAL = {Math.~Res.~Lett.},
  FJOURNAL = {Mathematical Research Letters},
    VOLUME = {22},
      YEAR = {2015},
    NUMBER = {2},
     PAGES = {403--416},
      ISSN = {1073-2780},
   MRCLASS = {11G05 (14K15)},
  MRNUMBER = {3342239},
MRREVIEWER = {John T.~Cullinan},
       DOI = {10.4310/MRL.2015.v22.n2.a4},
       URL = {https://doi.org/10.4310/MRL.2015.v22.n2.a4},
}

@article {BCS:2017,
    AUTHOR = {Bourdon, Abbey and Clark, Pete L. and Stankewicz, James},
     TITLE = {Torsion points on {CM} elliptic curves over real number
              fields},
   JOURNAL = {Trans. Amer. Math. Soc.},
  FJOURNAL = {Transactions of the American Mathematical Society},
    VOLUME = {369},
      YEAR = {2017},
    NUMBER = {12},
     PAGES = {8457--8496},
      ISSN = {0002-9947,1088-6850},
   MRCLASS = {11G05 (11G15)},
  MRNUMBER = {3710632},
MRREVIEWER = {Jie\ Shu},
       DOI = {10.1090/tran/6905},
       URL = {https://doi.org/10.1090/tran/6905},
}

@article {Brown:2012,
    AUTHOR = {Brown, Francis},
     TITLE = {Mixed {T}ate motives over {$\mathbb{Z}$}},
   JOURNAL = {Ann. of Math. (2)},
  FJOURNAL = {Annals of Mathematics. Second Series},
    VOLUME = {175},
      YEAR = {2012},
    NUMBER = {2},
     PAGES = {949--976},
      ISSN = {0003-486X,1939-8980},
   MRCLASS = {11S20 (11M32 14F42)},
  MRNUMBER = {2993755},
MRREVIEWER = {Pierre\ A.\ Lochak},
       DOI = {10.4007/annals.2012.175.2.10},
       URL = {https://doi.org/10.4007/annals.2012.175.2.10},
}

@article {Corvaja-Zannier:2004,
    AUTHOR = {Corvaja, P. and Zannier, U.},
     TITLE = {On integral points on surfaces},
   JOURNAL = {Ann. of Math. (2)},
  FJOURNAL = {Annals of Mathematics. Second Series},
    VOLUME = {160},
      YEAR = {2004},
    NUMBER = {2},
     PAGES = {705--726},
      ISSN = {0003-486X,1939-8980},
   MRCLASS = {11G35 (11G30 14G25)},
  MRNUMBER = {2123936},
MRREVIEWER = {Damien\ Roy},
       DOI = {10.4007/annals.2004.160.705},
       URL = {https://doi.org/10.4007/annals.2004.160.705},
}

@book {Cox:Primes,
    AUTHOR = {Cox, David A.},
     TITLE = {Primes of the form {$x^2+ny^2$}: {F}ermat, class field theory, and complex multiplication},
   EDITION = {3rd.~ed.},
      NOTE = {With contributions by R.~Lipsett},
 PUBLISHER = {AMS Chelsea Publishing, Providence, RI},
      YEAR = {2022},
     PAGES = {xv+533},
      ISBN = {978-1-470-47028-9},
   MRCLASS = {11A41 (11F11 11R11 11R16 11R18 11R37 11Y11)},
  MRNUMBER = {4502401},
}

@article {Daniels:2015,
    AUTHOR = {Daniels, Harris B. and Lozano-Robledo, \'{A}lvaro},
     TITLE = {On the number of isomorphism classes of {CM} elliptic curves
              defined over a number field},
   JOURNAL = {J. Number Theory},
  FJOURNAL = {Journal of Number Theory},
    VOLUME = {157},
      YEAR = {2015},
     PAGES = {367--396},
      ISSN = {0022-314X,1096-1658},
   MRCLASS = {14H52 (14G25 14H25 14K22)},
  MRNUMBER = {3373247},
MRREVIEWER = {Andrew\ Bremner},
       DOI = {10.1016/j.jnt.2015.05.009},
       URL = {https://doi.org/10.1016/j.jnt.2015.05.009},
}

@article {David:1999,
    AUTHOR = {David, Chantal and Kisilevsky, Hershy and Pappalardi,
              Francesco},
     TITLE = {Galois representations with non-surjective traces},
   JOURNAL = {Canad.~J.~Math.},
  FJOURNAL = {Canadian Journal of Mathematics. Journal Canadien de
              Mathèmatiques},
    VOLUME = {51},
      YEAR = {1999},
    NUMBER = {5},
     PAGES = {936--951},
      ISSN = {0008-414X},
   MRCLASS = {11F80 (11G05)},
  MRNUMBER = {1718676},
MRREVIEWER = {Chandrashekhar Khare},
       DOI = {10.4153/CJM-1999-041-0},
       URL = {https://doi.org/10.4153/CJM-1999-041-0},
}

@incollection {Diamond-Im:1995,
    AUTHOR = {Diamond, Fred and Im, John},
     TITLE = {Modular forms and modular curves},
 BOOKTITLE = {Seminar on {F}ermat's {L}ast {T}heorem ({T}oronto, {ON},
              1993--1994)},
    SERIES = {CMS Conf. Proc.},
    VOLUME = {17},
     PAGES = {39--133},
 PUBLISHER = {Amer. Math. Soc., Providence, RI},
      YEAR = {1995},
      ISBN = {0-8218-0313-1},
   MRCLASS = {11F11 (11F25 11G05 11G18)},
  MRNUMBER = {1357209},
}

@article {Faltings:1984,
    AUTHOR = {Faltings, G.},
     TITLE = {Erratum: ``{F}initeness theorems for abelian varieties over
              number fields''},
   JOURNAL = {Invent.~Math.},
  FJOURNAL = {Inventiones Mathematicae},
    VOLUME = {75},
      YEAR = {1984},
    NUMBER = {2},
     PAGES = {381},
      ISSN = {0020-9910},
   MRCLASS = {11D41 (11G30 14G25)},
  MRNUMBER = {732554},
MRREVIEWER = {James Milne},
       DOI = {10.1007/BF01388572},
       URL = {https://doi.org/10.1007/BF01388572},
}

@article {Faltings:1983,
    AUTHOR = {Faltings, G.},
     TITLE = {Endlichkeitssätze für abelsche {V}arietäten über
              {Z}ahlkörpern},
   JOURNAL = {Invent.~Math.},
  FJOURNAL = {Inventiones Mathematicae},
    VOLUME = {73},
      YEAR = {1983},
    NUMBER = {3},
     PAGES = {349--366},
      ISSN = {0020-9910},
   MRCLASS = {11D41 (11G30 14G25)},
  MRNUMBER = {718935},
MRREVIEWER = {James Milne},
       DOI = {10.1007/BF01388432},
       URL = {https://doi.org/10.1007/BF01388432},
}

@incollection {Fuchs:2014,
    AUTHOR = {Fuchs, Clemens},
     TITLE = {On some applications of {D}iophantine approximations [translation of \"{U}ber einege {A}nwendungen diophantischer {A}pproximationen]},
 BOOKTITLE = {On some applications of {D}iophantine approximations},
    SERIES = {Quad./Monogr.},
    VOLUME = {2},
     PAGES = {1--80},
 PUBLISHER = {Ed. Norm., Pisa},
      YEAR = {2014},
      ISBN = {978-88-7642-519-6},
   MRCLASS = {11G30 (01A75 11Jxx)},
  MRNUMBER = {3330349},
}

@article {Ihara:1986,
    AUTHOR = {Ihara, Yasutaka},
     TITLE = {Profinite braid groups, {G}alois representations and complex
              multiplications},
   JOURNAL = {Ann. of Math. (2)},
  FJOURNAL = {Annals of Mathematics. Second Series},
    VOLUME = {123},
      YEAR = {1986},
    NUMBER = {1},
     PAGES = {43--106},
      ISSN = {0003-486X,1939-8980},
   MRCLASS = {11G25 (14K22 20E18)},
  MRNUMBER = {825839},
MRREVIEWER = {David\ Goss},
       DOI = {10.2307/1971352},
       URL = {https://doi.org/10.2307/1971352},
}

@misc{Karpisz:2019,
  author  = {Karpisz, Ryan},
  title   = {Conditional bounds on heavenly elliptic curves over quadratic number fields},
  note    = {MA thesis: Wesleyan University, May 2019},
  year    = {2019}
}

@book {Lang:1987,
    AUTHOR = {Lang, Serge},
     TITLE = {Elliptic functions},
    SERIES = {Graduate Texts in Mathematics},
    VOLUME = {112},
   EDITION = {Second},
      NOTE = {With an appendix by J. Tate},
 PUBLISHER = {Springer-Verlag, New York},
      YEAR = {1987},
     PAGES = {xii+326},
      ISBN = {0-387-96508-4},
   MRCLASS = {11F03 (11G05 11G07 11G15 14G25)},
  MRNUMBER = {890960},
       DOI = {10.1007/978-1-4612-4752-4},
       URL = {https://doi.org/10.1007/978-1-4612-4752-4},
}

@article {Levin:2009,
    AUTHOR = {Levin, Aaron},
     TITLE = {Generalizations of {S}iegel's and {P}icard's theorems},
   JOURNAL = {Ann. of Math. (2)},
  FJOURNAL = {Annals of Mathematics. Second Series},
    VOLUME = {170},
      YEAR = {2009},
    NUMBER = {2},
     PAGES = {609--655},
      ISSN = {0003-486X,1939-8980},
   MRCLASS = {11J97 (32H25)},
  MRNUMBER = {2552103},
MRREVIEWER = {Min\ Ru},
       DOI = {10.4007/annals.2009.170.609},
       URL = {https://doi.org/10.4007/annals.2009.170.609},
}

@article {Levin:2016,
    AUTHOR = {Levin, Aaron},
     TITLE = {Integral points of bounded degree on affine curves},
   JOURNAL = {Compos. Math.},
  FJOURNAL = {Compositio Mathematica},
    VOLUME = {152},
      YEAR = {2016},
    NUMBER = {4},
     PAGES = {754--768},
      ISSN = {0010-437X,1570-5846},
   MRCLASS = {11G30 (14H25 30D35)},
  MRNUMBER = {3484113},
MRREVIEWER = {Joseph\ H.\ Silverman},
       DOI = {10.1112/S0010437X15007708},
       URL = {https://doi.org/10.1112/S0010437X15007708},
}

@misc{LMFDB:2024,
  shorthand    = {LMFDB},
  author       = {{The {LMFDB Collaboration}}},
  title        = {The {L}-functions and modular forms database},
  howpublished = {\url{https://www.lmfdb.org}},
  year         = {2024},
  note         = {[Online; accessed November 2024]},
}

@article {Lombardo:2018,
    AUTHOR = {Lombardo, Davide},
     TITLE = {On the uniform {R}asmussen-{T}amagawa conjecture in the {CM}
              case},
   JOURNAL = {Math. Res. Lett.},
  FJOURNAL = {Mathematical Research Letters},
    VOLUME = {25},
      YEAR = {2018},
    NUMBER = {6},
     PAGES = {1893--1910},
      ISSN = {1073-2780,1945-001X},
   MRCLASS = {14K15 (11F80 11G10)},
  MRNUMBER = {3934850},
MRREVIEWER = {Valentijn\ Zo\"{e}\ Karemaker},
       DOI = {10.4310/MRL.2018.v25.n6.a10},
       URL = {https://doi.org/10.4310/MRL.2018.v25.n6.a10},
}

@article {Lozano-Robledo:2022,
    AUTHOR = {Lozano-Robledo, {\'{A}}lvaro},
     TITLE = {Galois representations attached to elliptic curves with
              complex multiplication},
   JOURNAL = {Algebra Number Theory},
  FJOURNAL = {Algebra \& Number Theory},
    VOLUME = {16},
      YEAR = {2022},
    NUMBER = {4},
     PAGES = {777--837},
      ISSN = {1937-0652,1944-7833},
   MRCLASS = {11F80 (11G05 11G15 14H52)},
  MRNUMBER = {4467123},
MRREVIEWER = {Patrick\ Morton},
       DOI = {10.2140/ant.2022.16.777},
       URL = {https://doi.org/10.2140/ant.2022.16.777},
}

@article {MAGMA,
    AUTHOR = {Bosma, Wieb and Cannon, John and Playoust, Catherine},
     TITLE = {The {M}agma algebra system. {I}. {T}he user language},
      NOTE = {Computational algebra and number theory (London, 1993)},
   JOURNAL = {J. Symbolic Comput.},
  FJOURNAL = {Journal of Symbolic Computation},
    VOLUME = {24},
      YEAR = {1997},
    NUMBER = {3-4},
     PAGES = {235--265},
      ISSN = {0747-7171},
   MRCLASS = {68Q40},
  MRNUMBER = {MR1484478},
       DOI = {10.1006/jsco.1996.0125},
       URL = {http://dx.doi.org/10.1006/jsco.1996.0125},
}

@article {Mazur:78,
    AUTHOR = {Mazur, B.},
     TITLE = {Rational isogenies of prime degree (with an appendix by {D}.
              {G}oldfeld)},
   JOURNAL = {Invent. Math.},
  FJOURNAL = {Inventiones Mathematicae},
    VOLUME = {44},
      YEAR = {1978},
    NUMBER = {2},
     PAGES = {129--162},
      ISSN = {0020-9910,1432-1297},
   MRCLASS = {14K07 (10D35 14G25)},
  MRNUMBER = {482230},
MRREVIEWER = {V.\ V.\ Shokurov},
       DOI = {10.1007/BF01390348},
       URL = {https://doi.org/10.1007/BF01390348},
}

@article{McLeman-Rasmussen:2025,
       AUTHOR = {Cam McLeman and Christopher Rasmussen},
        TITLE = {Equivalence of conjectures on heavenly elliptic curves},
      JOURNAL = {Mathematika},
         YEAR = {to appear},
    SHORTHAND = {MR25},
       EPRINT = {2505.17474},
ARCHIVEPREFIX = {arXiv},
 PRIMARYCLASS = {math.NT},
          URL = {https://arxiv.org/abs/2505.17474},
}

@misc{McLeman-Rasmussen:2026,
  author       = {McLeman, Cam and Rasmussen, Christopher},
  title        = {Codebase for search for heavenly elliptic curves over quadratic fields},
  year         = {2026},
  howpublished = {\url{https://github.com/christopherrasmussen/hecqf-code}},
  note         = {GitHub repository}
}

@article {Ozeki:2013,
    AUTHOR = {Ozeki, Yoshiyasu},
     TITLE = {Non-existence of certain {CM} abelian varieties with prime
              power torsion},
   JOURNAL = {Tohoku Math. J. (2)},
  FJOURNAL = {The Tohoku Mathematical Journal. Second Series},
    VOLUME = {65},
      YEAR = {2013},
    NUMBER = {3},
     PAGES = {357--371},
      ISSN = {0040-8735,2186-585X},
   MRCLASS = {11G10 (14K22)},
  MRNUMBER = {3102540},
       DOI = {10.2748/tmj/1378991021},
       URL = {https://doi.org/10.2748/tmj/1378991021},
}

@article {Rasmussen-Tamagawa:2008,
    AUTHOR = {Rasmussen, Christopher and Tamagawa, Akio},
     TITLE = {A finiteness conjecture on abelian varieties with constrained
              prime power torsion},
   JOURNAL = {Math.~Res.~Lett.},
  FJOURNAL = {Mathematical Research Letters},
    VOLUME = {15},
      YEAR = {2008},
    NUMBER = {6},
     PAGES = {1223--1231},
      ISSN = {1073-2780},
   MRCLASS = {11G10 (11F80 11G15 14K15)},
  MRNUMBER = {2470396},
MRREVIEWER = {Pete L.~Clark},
       DOI = {10.4310/MRL.2008.v15.n6.a12},
       URL = {https://doi.org/10.4310/MRL.2008.v15.n6.a12},
}

@article{Rasmussen-Tamagawa:2017,
    AUTHOR = {Rasmussen, Christopher and Tamagawa, Akio},
     TITLE = {Arithmetic of abelian varieties with constrained torsion},
   JOURNAL = {Trans. Amer.~Math.~Soc.},
  FJOURNAL = {Transactions of the American Mathematical Society},
    VOLUME = {369},
      YEAR = {2017},
    NUMBER = {4},
     PAGES = {2395--2424},
      ISSN = {0002-9947},
   MRCLASS = {11G10 (11F80 14K15)},
  MRNUMBER = {3592515},
MRREVIEWER = {Davide Lombardo},
       DOI = {10.1090/tran/6790},
       URL = {https://doi.org/10.1090/tran/6790},
}

@article {Rasmussen-Tamagawa:2019,
    AUTHOR = {Rasmussen, Christopher and Tamagawa, Akio},
     TITLE = {Cyclic covers and {I}hara's question},
   JOURNAL = {Res. Number Theory},
  FJOURNAL = {Research in Number Theory},
    VOLUME = {5},
      YEAR = {2019},
    NUMBER = {4},
     PAGES = {Paper No. 33, 23},
      ISSN = {2522-0160,2363-9555},
   MRCLASS = {14H40 (11F80 11G10 11G30)},
  MRNUMBER = {4018545},
MRREVIEWER = {Benjamin\ Smith},
       DOI = {10.1007/s40993-019-0170-1},
       URL = {https://doi.org/10.1007/s40993-019-0170-1},
}

@article {Rubin-Silverberg:2009,
    AUTHOR = {Rubin, K. and Silverberg, A.},
     TITLE = {Point counting on reductions of {CM} elliptic curves},
   JOURNAL = {J. Number Theory},
  FJOURNAL = {Journal of Number Theory},
    VOLUME = {129},
      YEAR = {2009},
    NUMBER = {12},
     PAGES = {2903--2923},
      ISSN = {0022-314X,1096-1658},
   MRCLASS = {11G15 (11G05 11G07)},
  MRNUMBER = {2560842},
MRREVIEWER = {Jeffrey\ D.\ Achter},
       DOI = {10.1016/j.jnt.2009.01.020},
       URL = {https://doi.org/10.1016/j.jnt.2009.01.020},
}

@manual{SageMath,
  Shorthand    = {SM},
  Author       = {{The Sage Developers}},
  Title        = {{S}ageMath, the {S}age {M}athematics {S}oftware {S}ystem ({V}ersion 9.8)},
  note         = {Available at {\tt https://www.sagemath.org}},
  Year         = {2024},
}

@article {Serre:1972,
    AUTHOR = {Serre, Jean-Pierre},
     TITLE = {Propriètès galoisiennes des points d'ordre fini des courbes
              elliptiques},
   JOURNAL = {Invent.~Math.},
  FJOURNAL = {Inventiones Mathematicae},
    VOLUME = {15},
      YEAR = {1972},
    NUMBER = {4},
     PAGES = {259--331},
      ISSN = {0020-9910},
   MRCLASS = {14G25 (14K15)},
  MRNUMBER = {387283},
MRREVIEWER = {J. W. S. Cassels},
       DOI = {10.1007/BF01405086},
       URL = {https://doi.org/10.1007/BF01405086},
}

@article {Serre-Tate:1968,
    AUTHOR = {Serre, Jean-Pierre and Tate, John},
     TITLE = {Good reduction of abelian varieties},
   JOURNAL = {Ann. of Math. (2)},
  FJOURNAL = {Annals of Mathematics. Second Series},
    VOLUME = {88},
      YEAR = {1968},
     PAGES = {492--517},
      ISSN = {0003-486X},
   MRCLASS = {14.51},
  MRNUMBER = {236190},
MRREVIEWER = {M. J. Greenberg},
       DOI = {10.2307/1970722},
       URL = {https://doi.org/10.2307/1970722},
}

@article {Serre:1981,
    AUTHOR = {Serre, Jean-Pierre},
     TITLE = {Quelques applications du th\'eor\`eme de densit\'e{} de
              {C}hebotarev},
   JOURNAL = {Inst. Hautes \'Etudes Sci. Publ. Math.},
  FJOURNAL = {Institut des Hautes \'Etudes Scientifiques. Publications
              Math\'ematiques},
    NUMBER = {54},
      YEAR = {1981},
     PAGES = {123--201},
      ISSN = {0073-8301,1618-1913},
   MRCLASS = {12A75 (10D99 10H25 14G25)},
  MRNUMBER = {644559},
MRREVIEWER = {J.\ Tunnell},
       URL = {http://archive.numdam.org/article/PMIHES_1981__54__123_0.pdf},
}

@book {SGA:1972,
     TITLE = {Groupes de monodromie en gèomètrie algèbrique. {I}},
    SERIES = {Lecture Notes in Mathematics, Vol. 288},
      NOTE = {Sèminaire de Gèomètrie Algèbrique du Bois-Marie 1967--1969
              (SGA 7 I),
              Dirigè par A. Grothendieck. Avec la collaboration de M.
              Raynaud et D. S. Rim},
 PUBLISHER = {Springer-Verlag, Berlin-New York},
      YEAR = {1972},
     PAGES = {viii+523},
   MRCLASS = {14-06},
  MRNUMBER = {0354656},
 SHORTHAND = {SGA72},
}

@incollection {Sharifi:2002,
    AUTHOR = {Sharifi, Romyar T.},
     TITLE = {Relationships between conjectures on the structure of
              pro-{$p$} {G}alois groups unramified outside {$p$}},
 BOOKTITLE = {Arithmetic fundamental groups and noncommutative algebra
              ({B}erkeley, {CA}, 1999)},
    SERIES = {Proc. Sympos. Pure Math.},
    VOLUME = {70},
     PAGES = {275--284},
 PUBLISHER = {Amer. Math. Soc., Providence, RI},
      YEAR = {2002},
      ISBN = {0-8218-2036-2},
   MRCLASS = {11R32 (11R23 17B70 20F34)},
  MRNUMBER = {1935409},
MRREVIEWER = {Hiroaki\ Nakamura},
       DOI = {10.1090/pspum/070/1935409},
       URL = {https://doi.org/10.1090/pspum/070/1935409},
}

@article {Siegel:1929,
    AUTHOR = {Siegel, Carl L.},
     TITLE = {\"{U}ber einege {A}nwendungen diophantischer {A}pproximationen},
   JOURNAL = {Abh. Preuss. Akad. Wiss. Phys. Math.},
  FJOURNAL = {{A}bhandlungen der {P}reu\ss ischen {A}kademie der {W}issenschaften. {P}hysikalisch-mathematische {K}lasse},
    VOLUME = {1},
      YEAR = {1929},
     PAGES = {41--69},
}

@book {Silverman:1994,
    AUTHOR = {Silverman, Joseph H.},
     TITLE = {Advanced topics in the arithmetic of elliptic curves},
    SERIES = {Graduate Texts in Mathematics},
    VOLUME = {151},
 PUBLISHER = {Springer-Verlag, New York},
      YEAR = {1994},
      ISBN = {0-387-94328-5},
   MRCLASS = {11G05 (11G07 11G15 11G40 14H52)},
  MRNUMBER = {1312368},
MRREVIEWER = {Henri Darmon},
       DOI = {10.1007/978-1-4612-0851-8},
       URL = {https://doi.org/10.1007/978-1-4612-0851-8},
}

@book {Silverman:AEC,
    AUTHOR = {Silverman, Joseph H.},
     TITLE = {The arithmetic of elliptic curves},
    SERIES = {Graduate Texts in Mathematics},
    VOLUME = {106},
   EDITION = {Second},
 PUBLISHER = {Springer, Dordrecht},
      YEAR = {2009},
     PAGES = {xx+513},
      ISBN = {978-0-387-09493-9},
   MRCLASS = {11-02 (11G05 11G20 14H52 14K15)},
  MRNUMBER = {2514094},
MRREVIEWER = {Vasil\cprime \ \={I}.\ Andr\={\i}\u{\i}chuk},
       DOI = {10.1007/978-0-387-09494-6},
       URL = {https://doi.org/10.1007/978-0-387-09494-6},
}

@article {Suprunenko:1963,
    AUTHOR = {Suprunenko, D. A.},
     TITLE = {On the order of an element of a group of integral matrices},
   JOURNAL = {Dokl.~Akad.~Nauk BSSR},
  FJOURNAL = {Doklady Akademii Nauk BSSR},
    VOLUME = {7},
      YEAR = {1963},
     PAGES = {221--223},
      ISSN = {0002-354X},
   MRCLASS = {20.99},
  MRNUMBER = {0148727},
MRREVIEWER = {R. Kochendörffer},
}

@article {Tate-Oort:1970,
    AUTHOR = {Tate, John and Oort, Frans},
     TITLE = {Group schemes of prime order},
   JOURNAL = {Ann.~Sci.~Ècole Norm.~Sup.~(4)},
  FJOURNAL = {Annales Scientifiques de l'Ècole Normale Supèrieure. Quatrième
              Sèrie},
    VOLUME = {3},
      YEAR = {1970},
     PAGES = {1--21},
      ISSN = {0012-9593},
   MRCLASS = {14.50},
  MRNUMBER = {265368},
MRREVIEWER = {James Milne},
       URL = {http://www.numdam.org/item?id=ASENS_1970_4_3_1_1_0},
}

@article {Zarhin:1985,
    AUTHOR = {Zarhin, Yu. G.},
     TITLE = {A finiteness theorem for unpolarized abelian varieties over
              number fields with prescribed places of bad reduction},
   JOURNAL = {Invent.~Math.},
  FJOURNAL = {Inventiones Mathematicae},
    VOLUME = {79},
      YEAR = {1985},
    NUMBER = {2},
     PAGES = {309--321},
      ISSN = {0020-9910},
   MRCLASS = {14K15 (11G10)},
  MRNUMBER = {778130},
MRREVIEWER = {Gerd Faltings},
       DOI = {10.1007/BF01388976},
       URL = {https://doi.org/10.1007/BF01388976},
}

\end{CJK*}
\end{document}
